\documentclass[a4paper,12pt]{article}
\usepackage[utf8x]{inputenc}
\usepackage[english]{babel}
\usepackage{graphicx}
\usepackage{amscd}
\usepackage{amsfonts}
\usepackage{latexsym}
\usepackage{amssymb,amsmath}
\usepackage[usenames]{color}
\usepackage{caption}
\usepackage{subcaption}
\usepackage{amscd}
\usepackage{psfrag}

\usepackage{amsthm}

\theoremstyle{plain}
\newtheorem{teor}{Theorem}[section]
\newtheorem*{st}{Straightening Theorem}
\newtheorem*{mrmt}{Measurable Riemann Mapping Theorem}
\newtheorem{claim}{Claim}[section]

\newtheorem{prop}[teor]{Proposition}
\newtheorem{lemma}{Lemma}[section]
\newtheorem{remarks}{Remarks}[section]
\newtheorem{remark}{Remark}[section]
\newtheorem*{notation}{Notation}

\theoremstyle{definition}
\newtheorem{defi}[teor]{Definition}

\def\C{\mathbb{C}}
\def\S{\mathbb{S}}
\def\D{\mathbb{D}}
\def\R{\mathbb{R}}

\def\Z{\mathbb{Z}}

\def\H{\mathbb{H}}

\def\l0{_{\lambda_0}}
\def\i0{_{\iota_0}}

\title{ Parabolic-like mappings}
\author{Luna Lomonaco}

\begin{document}

\maketitle

\begin{abstract}
In this paper we introduce the notion of \textit{parabolic-like mapping}. Such an object is similar to
a polynomial-like mapping, but it has a parabolic external class, \textit{i.e.} an external map with a parabolic fixed
point.
We define the notion of parabolic-like mapping and we study the dynamical properties
of parabolic-like mappings.
We prove a Straightening Theorem for parabolic-like mappings which states that
any parabolic-like mapping of degree $2$
is hybrid conjugate to a member of the family
$$Per_1(1)= \left\{[P_A] \,| \, P_A(z)=z+ \frac{1}{z}+ A,\,\,A \in \C\right\},$$
a unique such member if the filled Julia set is connected.
\end{abstract}
\section{Introduction}

A \textit{polynomial-like map} of degree $d$ is a triple ($f,U',U$) where $U'$ and $U$ are open subsets of $\C$ isomorphic to discs,
$U'$ is compactly contained in $U$, and $f:U' \rightarrow U$ is a proper degree $d$ holomorphic map
(see \cite{DH}). A degree $d$ polynomial-like map is determined up to holomorphic conjugacy by its internal and external
classes, that is, the (conjugacy classes of the) maps which encode the dynamics of the polynomial-like map on the filled
Julia set and its complement. In particular the external class consists of degree
$d$ real-analytic orientation preserving and strictly expanding self-coverings of
the unit circle. The definition of a polynomial-like map captures the behaviour of a 
polynomial in a neighbourhood of its filled Julia set. By changing the external class
of a degree $d$ polynomial-like map with 
the external class of a degree $d$ polynomial (see \cite{DH}),
a degree $d$ polynomial-like map can be straightened to
a polynomial of the same degree.

In this paper we introduce a new object, a \textit{parabolic-like mapping}, similar to but different from a polynomial-like mapping.
The similarity resides in the fact that a parabolic-like map is a local concept, it is characterized by a filled Julia set and an external map, and
the external map of a degree $d$ parabolic-like mapping 
is a degree
$d$ real-analytic orientation preserving self-covering of
the unit circle. The difference resides in the fact that a parabolic-like map has a parabolic fixed point
with an attracting petal outside the filled Julia set, and the external map of a parabolic-like mapping has a parabolic fixed point.

The aim of this paper is to extend the theory of polynomial-like mappings (in the dynamical plane)
to parabolic-like mappings. 
Let us give an example which illustrates the class of maps we are considering.
The map $f_1(z)=z^2 + 1/4$ has a parabolic fixed point at $z=1/2$. Since
the parabolic basin of attraction of the parabolic fixed point resides in the interior of the
filled Julia set, while the repelling direction resides on the Julia set and outside of it, the external map
of $f_1(z)$
is hyperbolic. The map $f_1(z)$ presents polynomial-like restrictions.
On the other hand, let us interchange the roles of the filled Julia set and the closure of the
basin of attraction of infinity for $f_1$.
In other words, let us conjugate $f_1(z)$ by $\iota(z)= 1/z$ and obtain the map
$f_2(z)=\frac{4z^2}{4+z^2}$, and let us
define as filled Julia set for $f_2$ the closure of the
basin of attraction of the superattracting fixed point $z=0$. The basin of attraction of the parabolic fixed point $z=2$
now resides outside the filled Julia set, and gives rise to an the external class with a parabolic fixed point.
Appropriate restrictions of the map $f_2$ belong to the class of parabolic-like mappings.

As polynomial-like mappings are straightened to polynomials, we straighten degree $2$
parabolic-like mappings to members of a model family of maps with a parabolic external class.
We take as model family the family of quadratic rational maps with a parabolic fixed point of
multiplier $1$, normalized by fixing the parabolic fixed point to be infinity and the critical points to be $1$ and $-1$, this is
$$Per_1(1)=\{[P_A] \,| \, P_A(z)=z+ 1/z+ A,\,\,A \in \C\}.$$

\begin{figure}
  \centering
  \begin{minipage}{.4\textwidth}
  \centering
  \includegraphics[height= 5.5cm]{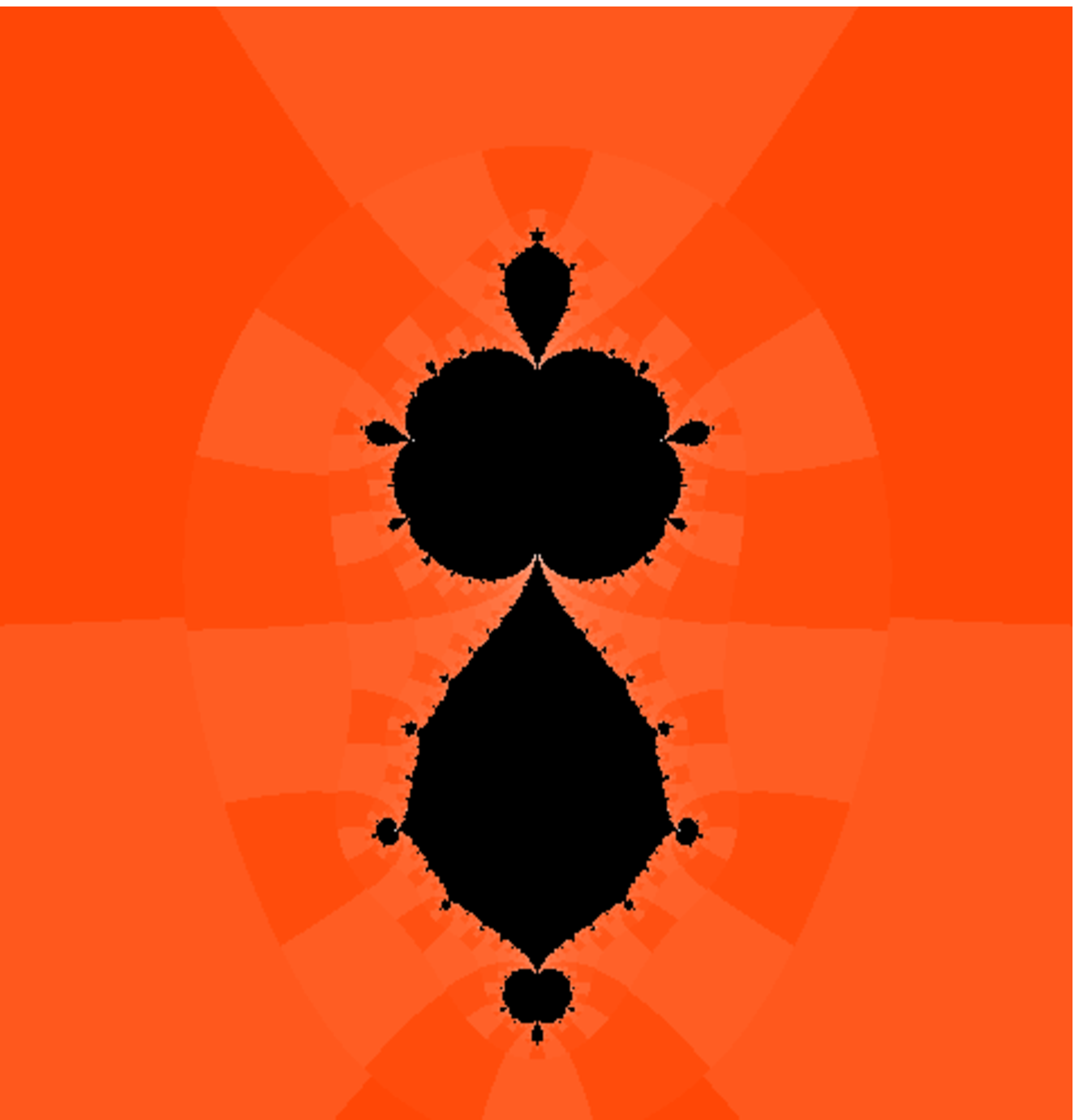}
  \captionof{figure}{\small Julia set of the map $C_{a}(z)=z^3+az^2+z$, $a=i$.}
  \label{Ca}
  \end{minipage} 
  \hspace{1.5cm}
  \begin{minipage}{.4\textwidth}
  \centering
  \includegraphics[height= 5.5cm]{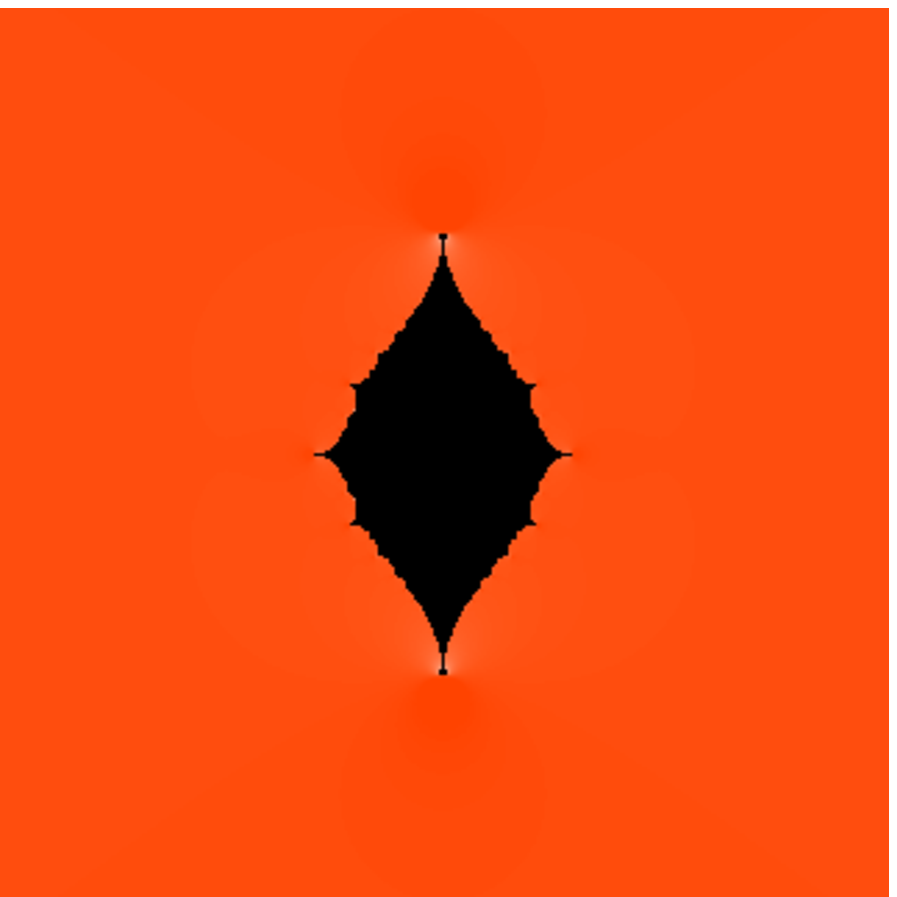}
  \captionof{figure}{\small Julia set of the map $P_1(z)=z+1/z+A$, $A=1$.}
  \label{M1}
  \end{minipage}
\end{figure}
All the maps in $Per_1(1)$ have a completely invariant Fatou component $\Lambda$, namely the parabolic basin of attraction of infinity.
We define the filled Julia set for these maps as 
$$K_A= \widehat \C \setminus \Lambda$$ (note that for every $A \neq 0$, $P_A$ has a unique
completely invariant Fatou component $\Lambda$,
hence $K_A$ is well defined, while for the map $P_0(z)=z+1/z$ we need to make a choice,
after which the filled Julia set $K_0$ is well defined).
The external class of this family is parabolic, and we prove in Proposition
\ref{extmap} that it is given by the class of $h_2(z)=\frac{3z^2+1}{3+z^2}$.

In this paper we will first define parabolic-like maps and the filled Julia set of a parabolic-like map.
Then we will construct and discuss the external class
in this setting. Finally, we will prove that we can straighten every degree $2$ parabolic-like map to a member of the family $Per_1(1)$,
by replacing the external map of the parabolic-like map by $h_2$ (see Fig. \ref{Ca} and \ref{M1}).\\

The author would like to thank her advisor, Carsten Lunde Petersen, for suggesting the idea of parabolic-like mapping,
and
for his help, support and encouragement. This paper was written during the author's Ph.d. Hence the
author would like to thank Roskilde University and Universit\'e Paul Sabatier for their hospitality, and Roskilde
University, the ANR-08-JCJC-0002 founded by the Agence Nationale de la Recherche and the Marie Curie RTN 035651-CODY for
their financial support during her Ph.d. 
\section{Preliminaries}\label{pre}
In this paper we are studying restrictions of maps with a parabolic fixed point of multiplier $1$.
By a change of coordinates we can consider the parabolic fixed point to be at $z=0$, hence we will
consider maps of the form $$f(z)=z(1+az^n+...),\,\,\,n \geq 1,\,\,a \neq 0.$$
The integer $n$ is the
\textit{degeneracy/parabolic multiplicity} of the parabolic fixed point. In a neighborhood of
a parabolic fixed point of parabolic multiplicity $n$,
there are $n$ attracting petals,
which alternate with $n$ repelling petals (for the definition of petal see \cite{Sh} or \cite{M}).
We will denote the petals by $\Xi$.
On each petal $\Xi$ there exists a conformal map which conjugates the map $f$ to a translation (see \cite{Sh} or \cite{M}).
This map is called a \textit{Fatou coordinate} for the petal $\Xi$, and it is unique up to composition with a translation.
We will denote Fatou coordinates by $\phi$.
Often it is convenient to consider the quotient of a petal $\Xi$ under the equivalence relation identifying $z$
and $f(z)$ if both $z$ and $f(z)$ belong to $\Xi$. This quotient manifold is called the \textit{\'Ecalle
cilinder}, and it is conformally isomorphic to the infinite cylinder $\C / \Z$ (see \cite{Sh} and
\cite{M}). \\

The Straightening Theorem is obtained by surgery, applying the Measurable Riemann Mapping Theorem (stated below). For
a proof of the Measurable Riemann Mapping Theorem and the notion
of almost complex structure, quasiconformal mappings and quasisymmetric mappings, the reader
is referred to \cite{Ah} or, for a modern treatment, to \cite{Hu}.
\begin{mrmt}
 Let $\sigma$ be a bounded almost complex structure on a domain $U\subset \C$.
 Then there exists a quasiconformal homeomorphism $\varphi: U \rightarrow \C$ such that
 $$\sigma= \varphi^* \sigma_0.$$
\end{mrmt}
\begin{notation}
 We will use the following notation:
 $$ \H_l=\{z \in \C | \mbox{ Re}(z)< 0 \},$$
 $$ \H_r=\{z \in \C | \mbox{ Re}(z)> 0 \}.$$
\end{notation}

\section{Definitions and statement of \\the Straightening Theorem}\label{de}
A parabolic-like map is an object introduced to extend the notion of polynomial-like maps to maps with a parabolic external map.
The domain of a parabolic-like map is not 
contained in the range, and the set of points with infinite
forward orbit is
not contained in the intersection of the domain and the range. This calls for a partition of the set of points with infinite
forward orbit
into a filled
Julia set compactly contained
in both domain and range and exterior attracting petals.
\begin{figure}[hbt!]
\centering
\psfrag{A}{$\Omega$}
\psfrag{A'}{$\Delta'$}
\psfrag{gamma+}{$\gamma_+$}
\psfrag{gamma-}{$\gamma_-$}
\psfrag{AA}{$\Omega'$}
\psfrag{d}{$f:U' \stackrel{d:1}{\rightarrow}U$}
\psfrag{1}{$f:\Delta' \stackrel{1:1}{\rightarrow} \Delta$}
\psfrag{D}{$\Delta$}
\includegraphics[width= 13cm]{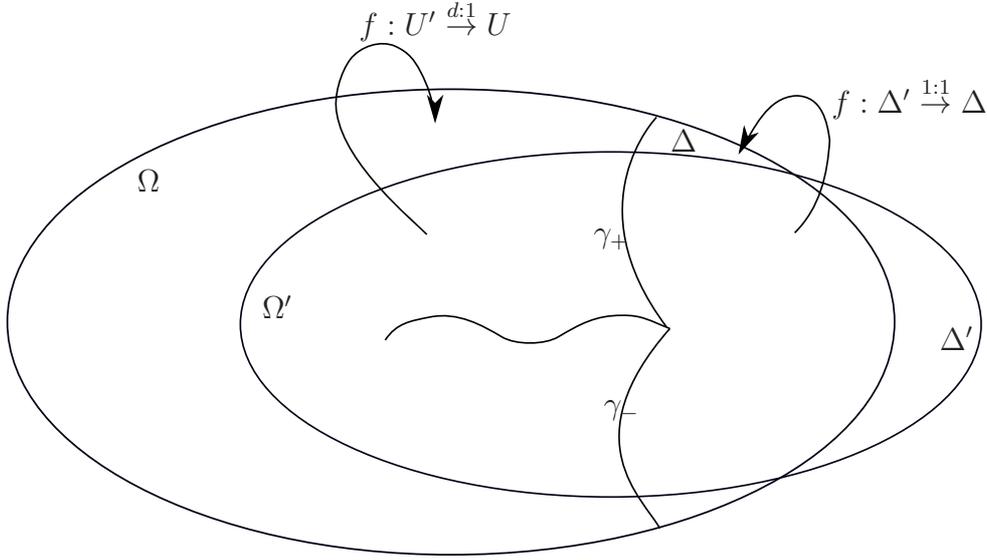}
\caption{\small On a parabolic-like map ($f,U',U,\gamma$) the arc $\gamma$ divides $U'$ and $U$
into $\Omega', \Delta'$
 and $\Omega, \Delta$ respectively. These sets are such that $\Omega'$ is compactly contained in $U$, $\Omega'\subset
\Omega$, $f:\Delta' \rightarrow \Delta$ is an isomorphism and $\Delta'$ contains at least one attracting fixed petal of the parabolic
fixed point.}
\label{AAA}
\end{figure}

\begin{defi}\label{definitionparlikemap} \textbf{(Parabolic-like maps)\,\,\,}
A \textit{parabolic-like map} of degree $d\geq2$ is a 4-tuple ($f,U',U,\gamma$) where 
\begin{itemize}
	\item $U'$ and $U$ are open subsets of $\C$, with $U',\,\, U$ and $U \cup U'$ isomorphic to a disc, and $U'$ not
contained in $U$,
	\item $f:U' \rightarrow U$ is a proper holomorphic map of degree $d\geq 2$ with a parabolic fixed point at $z=z_0$ of
 multiplier 1,
	\item $\gamma:[-1,1] \rightarrow \overline {U}$ is an arc with $\gamma(0)=z_0$, forward invariant under $f$, $C^1$
on $[-1,0]$ and on $[0,1]$, and such
that
$$f(\gamma(t))=\gamma(dt),\,\,\, \forall -\frac{1}{d} \leq t \leq \frac{1}{d},$$
$$\gamma([ \frac{1}{d}, 1)\cup (-1, -\frac{1}{d}]) \subseteq U \setminus U',\,\,\,\,\,\,\gamma(\pm 1) \in \partial U.$$
It resides in repelling petal(s) of $z_0$ and it divides $U'$ and $U$ into $\Omega', \Delta'$ and $\Omega, \Delta$
respectively, such that $\Omega' \subset \subset U$ 
(and $\Omega' \subset \Omega$), $f:\Delta' \rightarrow \Delta$ is an isomorphism (see Fig. \ref{AAA}) and
$\Delta'$ contains at least one attracting fixed petal of $z_0$. We call the arc $\gamma$ a \textit{dividing arc}.

\end{itemize}

\end{defi}
\begin{notation}\label{gamma}
We can consider $\gamma=\gamma_+\cup \gamma_-$, where $\gamma_{+}:[0,1]\rightarrow
\overline {U}, \,\,\,\gamma_{-}:[0,-1]\rightarrow \overline {U}, \,\,\,
\gamma_{\pm}(0)=z_0$. Where it will be convenient (e.g. in the examples) we will refer to $\gamma_{\pm}$ instead of
$\gamma$.
\end{notation}

\subsubsection*{Examples}\label{ex}
\begin{enumerate}
	\item Consider the function $h_2(z)=\frac{3z^2+1}{3+z^2}$. This map has critical points at $z=0$ and at $\infty$,
	and a parabolic fixed point at $z=1$ of multiplier $1$ and
parabolic multiplicity $2$. The attracting directions of the parabolic fixed point are along the real axis,
while the repelling ones are
perpendicular to the real axis. The repelling petals $\Xi_{+}$ and $\Xi_-$ intersect the unit circle and can be taken to be
reflection symmetric around the unit circle, since $h_2$ is autoconjugate by the reflection $T(z)=\frac{1}{\bar{z}}$.
Let
$\phi_{\pm}: \Xi_{\pm} \rightarrow \H_l  $ be Fatou coordinates with axis tangent to the unit circle at the parabolic
fixed point.
The image of the unit circle in the Fatou coordinate planes are horizontal lines,
which we can suppose coincide with $\R_-$, possibly changing the normalizations
of $\phi_{\pm}$.
Choose $\epsilon >0$ and
define
$U'=\{ z : |z|< 1 + \epsilon \}$, and $U=h_2(U')$.
Let $z_{\pm}$ be intersection points of $\Xi_{\pm}$ respectively and $\partial U$. Thus
$\phi_{+}(z_{+})= m_{+}$ with $Im(m_+)<0$, and $\phi_{-}(z_{-})= m_{-}$ with $Im(m_-) >0$.
Define the dividing arcs as:
$$\gamma_{+}: [0,1] \rightarrow U\,\,\,\,\,\,\,\,\,\,\,\,\,\,\,\,\,\,\,\,\,\,\,\,\,\,\,\,\,\,\,\gamma_{-}: [0,-1] \rightarrow U$$
$$t \rightarrow \phi_{+}^{-1}(log_d(t)+ m_+),\,\,\,\,\,\,\,\,\,\,\,\,\,t \rightarrow \phi_{-}^{-1}(log_d(-t)+m_-).$$ 
Then ($h_2,U',U,\gamma$) is a parabolic-like map of degree $2$.

\begin{figure}[hbt!]
\centering
\psfrag{gp}{$\gamma_+$}
\psfrag{g-}{$\gamma_-$}
\psfrag{U}{$U$}
\psfrag{U'}{$U'$}
\psfrag{O}{$\Omega$}
\psfrag{O'}{$\Omega'$}
\psfrag{oa}{$\mathcal{A}(0)$}
\psfrag{-a}{$c$}
\psfrag{-a/2}{$0$}
\psfrag{a}{$s$}
\psfrag{f-1w}{$\varphi^{-1}(w)$}
\psfrag{f-1wb}{$\varphi^{-1}(\overline{w})$}
\psfrag{fi}{$\varphi$}
\psfrag{f-a}{$\varphi(c)$}
\psfrag{1/3}{$1/3$}
\psfrag{w}{$w$}
\psfrag{wb}{$\overline{w}$}
\psfrag{fa}{$\varphi(0)$}
\includegraphics[width= 12cm]{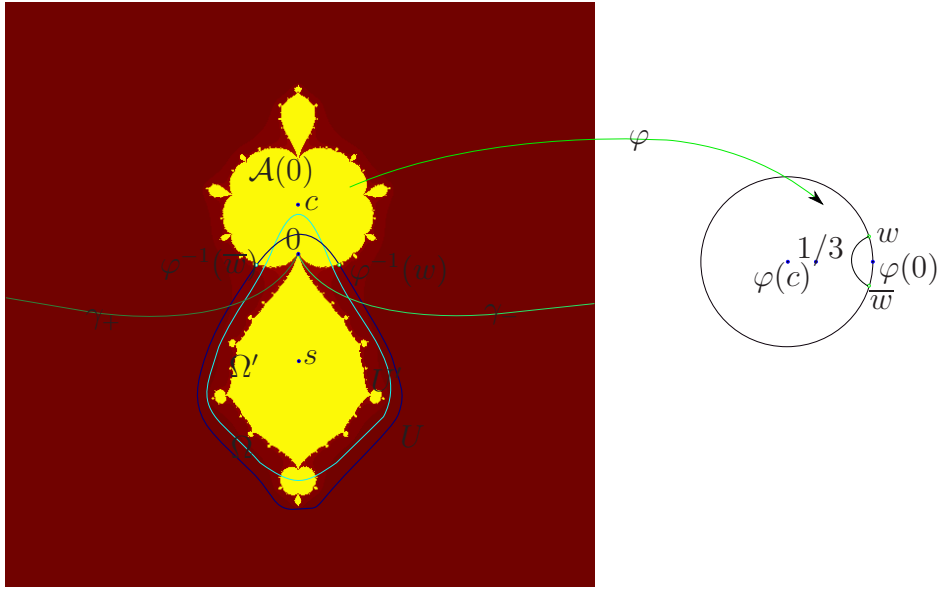}
\caption{\small Construction of a degree $2$ parabolic-like map from the map\\ $(C_a(z)=z+az^2+z^3)$, for $a=i$. The superattracting
fixed point $z=\frac{-a - \sqrt{a^2 -3}}{3} $ is denoted by $s$, and the critical point $z=\frac{-a + \sqrt{a^2-3}}{3}$ in the basin of attraction of the parabolic
fixed point is denoted by $c$.}
\label{cubicbcnnew}
\end{figure}
	\item Let $(C_a(z)=z+az^2+z^3)$, for $a=i$. This map has a superattracting fixed point $s$ at $z=\frac{-a - \sqrt{a^2 -3}}{3} $,
	a critical point $c$ at $z=\frac{-a + \sqrt{a^2-3}}{3}$
and a parabolic fixed point at $z=0$ with multiplier and parabolic multiplicity $1$. Call
$\mathcal{A}(0)$
the immediate basin of attraction of the parabolic fixed point. Then the critical point $c$ belongs to $\mathcal{A}(0)$.
Let $\varphi : \mathcal{A}(0) \rightarrow \D$ be
the Riemann map normalized by setting $\varphi(c)=0$ and $\varphi(z)\stackrel{z\rightarrow 0}\longrightarrow 1$, and let
$\psi:\D \rightarrow \mathcal{A}(0)$ be its
inverse. By the Carathéodory Theorem the map $\psi$ extends
continuously to $\S^1$. Note that $\varphi  \circ f \circ
\psi=h_2$. Let $w$ be an $h_2$ periodic point in the first
quadrant, such that the hyperbolic geodesic $\widetilde{\gamma} \in \D$ connecting $w$ and $\overline{w}$ separates the
critical value $z=1/3$ from the parabolic fixed point $z=1$. Let $U$ be the Jordan domain bounded by $\widehat{\gamma}=
\psi(\widetilde{\gamma})$, union the arcs up to potential level $1$ of the external rays landing at $\psi(w)$ and
$\psi(\overline{w})$, together with the arc of the level $1$ equipotential connecting this two rays around $s$ (see
Fig. \ref{cubicbcnnew}). Let
$U'$ be the connected component of $f^{-1}(U)$ containing $0$ and the dividing arcs $\gamma_{\pm}$ be the fixed external rays landing at the
parabolic
fixed point
$0$ and parametrized by potential. Then  ($f,U',U,\gamma$) is a parabolic-like map of degree $2$ (see Fig.
\ref{cubicbcnnew}).

	\item Let $f(z)=z^2+c$, for $c=(-1+3\sqrt{3}i)/8$ (fat rabbit). Its third iterate $f^3$ has a
parabolic fixed point at $z=(-1+\sqrt{3}i)/4$ of multiplier $1$ and parabolic multiplicity $3$.
Let $\mathcal{A}_0$ be the component of the
immediate basin of attraction of the 
parabolic fixed point containing $z=0$. Number the connected components of the immediate attracting basin in the dynamical order (which
here is the counterclockwise direction around $a$). Let $\varphi :\mathcal{A}_0 \rightarrow \D$ be the Riemann map, normalized by
$\varphi(0)=0$ and $\varphi(z)\stackrel{z\rightarrow a}\longrightarrow 1$, and let $\psi:\D \rightarrow \mathcal{A}_0$ be its
inverse. The map $\psi$ extends continuously to $\S^1$, and $\varphi \circ f^3 \circ \psi=h_2$. As above let $w$ be a $h_2$
periodic point in the
first quadrant such that the hyperbolic geodesic $\widetilde{\gamma}$ connecting $w$ and $\overline{w}$ separates the
critical value $z=1/3$ from the
parabolic fixed point $z=1$. Define $\widehat{\gamma}= \psi(\widetilde{\gamma})$ and $\widehat{\gamma'}=
f^{-1}(\widehat{\gamma})\cap \overline{\mathcal{A}_2}$. Let $U$ be the Jordan domain bounded by $\widehat{\gamma}$ union the
arcs up to potential level 1 of
the external rays landing at $\psi(w)$ and $\psi(\overline w)$ union $\widehat{\gamma'}$ union the arcs up to
potential level $1$ of the external rays landing at $f^{-1}(\psi(w))\cap \overline{\mathcal{A}_2}$ and $f^{-1}(\psi(\overline
w))\cap \overline{\mathcal{A}_2}$, together
with the two arcs of the level $1$ equipotential connecting this four rays around the parabolic fixed point. Let $U'$ be the
connected component of $f^{-3}(U)$ containing $(-1+\sqrt{3}i)/4 $
and the dividing arcs $\gamma_+$ and $\gamma_-$ be the external rays for angles $1/7$ and $2/7$
respectively parametrized by potential. Then 
($f^3,U',U,\gamma$) is a parabolic-like map of degree $2$ (see Fig. \ref{fatrabbit}).\\

More generally, define $\lambda_{p/q}= \exp(2\pi i p/q)$ with $p$ and $q$ coprime, $c_{p/q}=
\frac{\lambda_{p/q}}{2}-\frac{\lambda^2_{p/q}}{4}$ and consider $f_q= z^2 + c_{p/q}$. The map $f_q$ has a parabolic
fixed point of multiplier $\lambda_{p/q}$ at $z=\lambda_{p/q}/2$, therefore $f^q$ has a parabolic fixed point of
multiplier $1$ and
parabolic multiplicity $q$. 

Repeating the construction done above one can see that the map $f^q$ restricts to a degree $2$
parabolic-like map.

\begin{figure}[hbt!]
\centering
\psfrag{g+}{$\gamma_+$}
\psfrag{g-}{$\gamma_-$}
\psfrag{u}{$U$}
\psfrag{u'}{$U'$}
\psfrag{f}{$f^3$}
\includegraphics[width= 12cm]{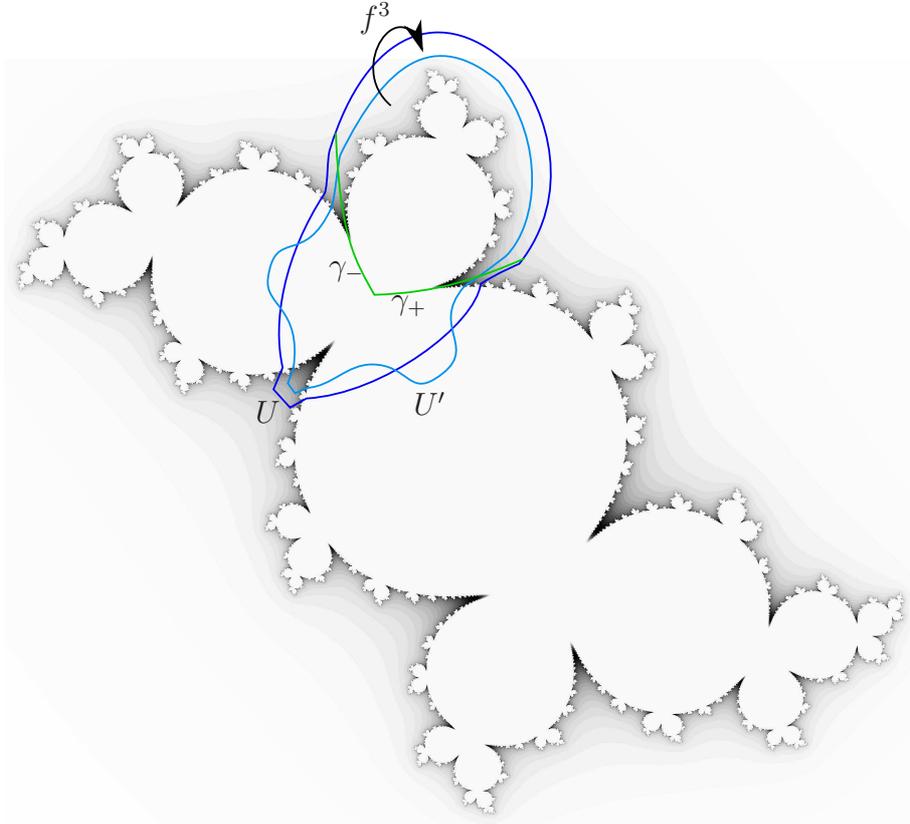}
\caption{\small The third iterate of the map $f=z^2+c$, for $c=(-1+3\sqrt{3}i)/8$, restricts to a degree $2$ parabolic-like map.}
\label{fatrabbit}
\end{figure}	
	
\end{enumerate}
\begin{defi}
Let $(f, U', U, \gamma)$ be a parabolic-like map. We define the \emph{filled Julia set $K_f$} of $f$ as the set
of points
in $U'$ that
 never leave $(\Omega' \cup \gamma_{\pm}(0))$ under iteration:
$$K_{f}:=\{z\in U'\,\vert\, \forall n \geq 0\ ,\,\, f^{n}(z)\in \Omega' \cup \gamma_{\pm}(0) \}. $$
\end{defi}

\begin{remark}
An equivalent definition for the filled Julia set of $f$ is 
$$K_f= \bigcap_{n \geq 0} f^{-n}(U\setminus \Delta).$$ 
The filled Julia set is a compact subset of $U \cap U'$ and it is full (since it is the intersection of topological disks).
\end{remark}
As for polynomials, we define the Julia set of $f$ as the boundary of the filled Julia set:
$$J_{f}:=\partial K_{f}$$

\subsubsection*{Motivations for the definition}
A parabolic-like map can be seen as the union of two different dynamical parts: a polynomial-like part
(on $\Omega'$) and a parabolic one (on $\Delta'$), which are connected by the dividing arc $\gamma$.

The parabolic fixed point belongs to the interior of the domain of a parabolic-like map in order
to insure that the filled Julia set is compactly contained in the intesection of the domain and the range.
The
dividing arc separates the exterior attracting petals from the filled Julia set of the parabolic-like mapping,
and for this reason the dividing arc is part of the definition
of parabolic-like mapping (note that we could have constructed the dividing arc a posteriori by Fatou coordinates).
The definition of parabolic-like map also guarantees the existence of an annulus, $U \setminus \Omega'$, essential in defining the external class and
to perform the surgery which will
give
the Straightening Theorem.

There are many prospect definitions of a parabolic-like map. The one introduced here is flexible enough to capture many
interesting examples, and rigid enough to allow for a viable theory.
\subsubsection*{Conjugacies and statement of the main result}
We say that
$(f,U'_1,U_1,\gamma_1)$ is a \textit{parabolic-like restriction} of $(f,U'_2,U_2,\gamma_2)$ if $U'_1 \subseteq U'_2$ and
$(f,U'_i,U_i,\gamma_i),\,\,i=1,2$ are parabolic-like maps
with the same degree and filled Julia set.
\begin{defi}\textbf{(Conjugacy for parabolic-like
mappings)\,\,\,\,\,\,\,\,\,\,\,\,\,\,\,\,\,\,\,\,\,\,\,\,\,\,\,\,\,\,\,\,\,\,\,\,\,\,\,\,\,\,\,\,\,}
We say that the parabolic-like mappings ($f,U',U,\gamma_{f}$) and ($g,V',V,\gamma_{g}$)
are \textit{topologically conjugate} if there exist parabolic-like restrictions 
($f,A',A,\gamma_{f}$) and ($g,B',B,\gamma_{g}$), and a homeomorphism $\varphi :A \rightarrow B$
such that
$\varphi(\gamma_{\pm f})=\gamma_{\pm g}$ and
$$\varphi(f(z))=g(\varphi(z))\,\,\,\,\mbox{ on } \Omega'_{A_f} \cup \gamma_{ f}$$
If moreover $\varphi$ is quasiconformal (and $\bar{\partial}\varphi = 0$ a.e. on $K_f$), we say that $f$ and $g$
 are \textit{quasiconformally} (\textit{hybrid}) conjugate.
\end{defi}
A topological conjugacy between parabolic-like maps is a homeomorphism
defined on a neighborhood of the filled Julia set, which conjugates dynamics just on $ \Omega'\cup \gamma$. This definition
allows flexibility regarding the parabolic multiplicity of the parabolic fixed point.

In this paper we will prove the following:
\begin{st}\label{tst}
\
\begin{enumerate}
\item Every degree $2$ parabolic-like mapping ($f,U',U,\gamma_{f}$) is hybrid equivalent to a member of the family
$Per_1(1)$.
\item Moreover, if $K_f$ is connected, this member is unique. 
\end{enumerate}
\end{st}
Part $1$ follows from Proposition \ref{stex} together with Theorem \ref{thm}, while part $2$ follows from Proposition \ref{unic}.
\subsection{Equivalence of parabolic-like mappings and\\ Isotopy}

Two parabolic-like maps are \textit{equivalent}, and we do not distinguish between them, if they have a common parabolic-like restriction.
Given a parabolic-like map $(f,U'_1,U_1,\gamma_1)$, the arc $\gamma_2: [-1,1]\rightarrow \overline U$
with $\gamma_2(0)= \gamma_1(0)$ is \textit{isotopic} to $\gamma_1$ if
there exists a domain $U_2' \subseteq U'_1$ for which
 $(f,U'_i,U_i,\gamma_i),\,\,i=1,2 $ have a common parabolic-like restriction.

\begin{lemma}\label{is}
 Let ($f,U',U,\gamma$) be a parabolic-like map, and let $\gamma_s:[-1,1]\rightarrow \overline U$  be an arc
forward invariant under $f$, with $\gamma_s(0)= \gamma(0)$ and $C^1$ on $[-1,0]$ and $[0,-1]$.
Then $\gamma_s$ and $\gamma$ are isotopic if and only if
their projections to \'Ecalle cylinders are isotopic and the isotopies are disjoint from the projections
of the filled Julia set and the critical points.
\end{lemma}
\begin{proof}
Let us prove that, if the projections of
$\gamma$ and $\gamma_s$ to \'Ecalle cylinders are isotopic
and the isotopies are disjoint from the projections
of the filled Julia set and the critical points, then $\gamma_s$ and $\gamma$ are isotopic. The vice versa is trivial. 

\begin{figure}[hbt!]
\centering
\psfrag{g}{$\gamma_+$}
\psfrag{g-}{$\gamma_-$}
\psfrag{b}{$\gamma_s$}
\psfrag{r}{$\gamma_{\hat s}$}
\psfrag{f}{$f$}
\psfrag{H}{$H_+$}
\psfrag{bo}{$H_+(s,\cdot)$}
\psfrag{ro}{$H_+(\hat s,\cdot)$}
\psfrag{go}{$H_+(0,\cdot)$}
\includegraphics[width= 12cm]{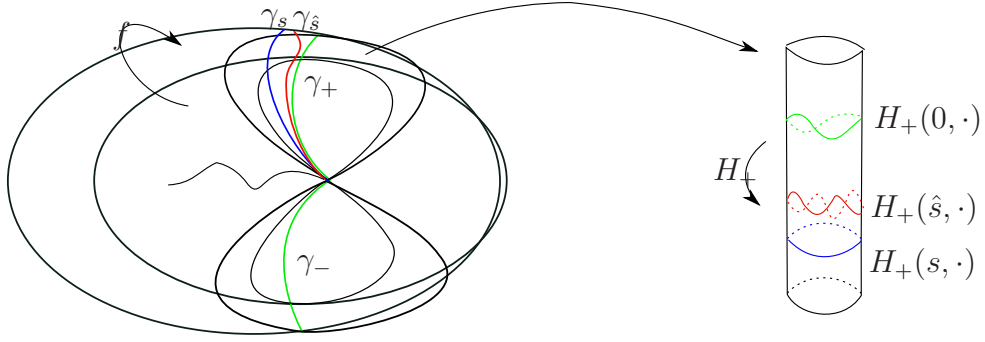}
\caption{\small Construction of dividing arcs isotopic to $\gamma$.}
\label{isotopy}
\end{figure}

Let $\Xi_+$ and $\Xi_-$ be repelling petals where $\gamma_+$ and $\gamma_-$ respectively reside
(note that $\Xi_+$ and $\Xi_-$ may coincide). Then the quotient 
manifolds $\Xi_+ /f$, $\Xi_- /f$ are
conformally isomorphic to the bi-infinite cylinder.
Call $\beta$ the isomorphism between $\Xi_+/f$ and $\C/\Z$, and $\delta$ the 
isomorphism between $\Xi_-/f$ and $\C/\Z$. Let 
$$H_+
:[0,1] \times [\tau,d\tau] \rightarrow
\C / \Z$$ 
$$(s,t) \rightarrow H_+(s,t), $$
$$H_-
:[0,1] \times [d \hat \tau, \hat \tau]  \rightarrow
\C / \Z$$ 
$$(s,t) \rightarrow H_-(s,t), $$
 be isotopies, disjoint from the projections of the filled Julia set and the critical points,
such that for every fixed $s \in
[0,1],$ both $H_{\pm}(s,t):\R/\Z \rightarrow \C /\Z$ are at least $C^1$. Set 
$\gamma_{s +}[\tau,d\tau]=\beta^{-1}(H_+(s, \cdot))$ and 
$\gamma_{s -}[d \hat \tau, \hat \tau]=\delta^{-1}(H_-(s, \cdot))$
Define $\gamma_s$ by extending $\gamma_{s +}$ and $\gamma_{s -}$
 by the dynamics of $f$ to forward
invariant curves
 in $\Xi_+$ and $\Xi_-$ respectively (see Picture \ref{isotopy}), \textit{i.e.}:
\begin{enumerate}
 \item $\gamma_{s +}(d^nt)=f^n(\gamma_{s +}(t)),\,\,\gamma_{s +}(t/d^n)$ 
$=f(\gamma_{s +}(t))^{-n} \,\,\, \forall \tau
\leq t \leq d \tau$;
 \item $\gamma_{s -}(d^nt)=f^n(\gamma_{s -}(t)),\,\,\gamma_{s -}(t/d^n)$
$=f(\gamma_{s -}(t))^{-n} \,\,\, \forall d \hat
\tau \leq t \leq \hat \tau$;
\item $\gamma_s(\pm 1) \in \partial U$ and $\gamma_s(0)=\gamma(0);$
\end{enumerate}
where $f(\gamma_{s})^{-n}$ is the branch which gives
 continuity. Then $\gamma_s$ divides $U$ and $U'$ in $\Omega_s,\,\,\Delta_s$ and 
$\Omega'_s,\,\,\Delta'_s$ respectively, and by construction $\Omega'_s$ contains $K_f$ and all
the critical points of $(f,U',U,\gamma)$.
Hence $(f,U',U,\gamma_s)$ is a parabolic-like restriction of $(f,U',U,\gamma)$, and
thus the arcs $\gamma$ and $\gamma_s$ are isotopic.
\end{proof}

Note that, by construction, if ($f,U',U,\gamma$) is a parabolic-like map and $\gamma_s$ is isotopic to $\gamma$,
then the arc
$\gamma_{+
s}$ resides in
the same petal as $\gamma_+$ and the arc $\gamma_{- s}$ resides in the same petal as $\gamma_-$.
\section{The external class of a parabolic-like map}\label{c}
 In analogy with the polynomial-like setting, we want to associate to any parabolic-like map ($f,U',U,\gamma $)
of degree $d$ a real-analytic map 
$h_f : \S^1 \rightarrow \S^1$ of the same degree $d$ and with a parabolic fixed point, unique up to conjugacy by
a real-analytic diffeomorphism. We will
call $h_f$ an \textit{external map} of $f$, and we will call $[h_f]$ (its
conjugacy class under real-analytic diffeomorphisms) the \textit{external class} of $f$. 
\subsubsection*{Construction of an external map of a parabolic-like map $f$ with connected Julia set}\label{em}
The construction of an external map of a parabolic-like map with connected Julia set follows the construction of an
external map in \cite{DH}, up to the differences given by the geometry of our setting.
Let ($f,U',U,\gamma$) be a parabolic-like map of degree $d$ with connected filled Julia set $K_f$. Then
$K_f$ contains all the critical points of $f$ and hence $f:U'\setminus K_f
\rightarrow U\setminus K_f$ is a holomorphic degree $d$ covering map.
Let $$\alpha : \widehat{\C} \setminus K_f \longrightarrow \widehat{\C} \setminus
\overline{\D}\,\,\,\,\,\,\,\,\,\,\,\,\, (1)$$
be the Riemann map, normalized by $\alpha(\infty)=\infty$ and $\alpha(\gamma(t)) \rightarrow 1$ as $t \rightarrow 0$.
Write $W'=\alpha(U' \setminus K_f)$ and $W=\alpha(U \setminus K_f)$ (see Fig. \ref{ABB}) and define the map:
$$h^+:= \alpha \circ f \circ \alpha^{-1}:\,\,W'  \rightarrow W,$$
Then the map $h^+$ is a holomorphic degree $d$ covering. 
\begin{figure}[hbt!]
\centering
\psfrag{A}{$U$}
\psfrag{A'}{$U'$}
\psfrag{B'_+}{$W'$}
\psfrag{B_+}{$W$}
\psfrag{B}{$W$}
\psfrag{B'}{$W'$}
\psfrag{h}{$h^+$}
\psfrag{f}{$f$}
\psfrag{pr}{$\alpha$}
\psfrag{S}{$\S_1$}
\psfrag{g+}{$\gamma_+$}
\psfrag{g-}{$\gamma_-$}
\psfrag{pg+}{$\widetilde{\gamma_+}$}
\psfrag{pg-}{$\widetilde \gamma_-$}
\psfrag{ag+}{$\alpha(\gamma_+)$}
\psfrag{ag-}{$\alpha(\gamma_-)$}
\psfrag{AA}{$Q_f$}
\psfrag{D}{$\Delta'$}
\psfrag{kf}{$K_f$}
\includegraphics[width= 15cm]{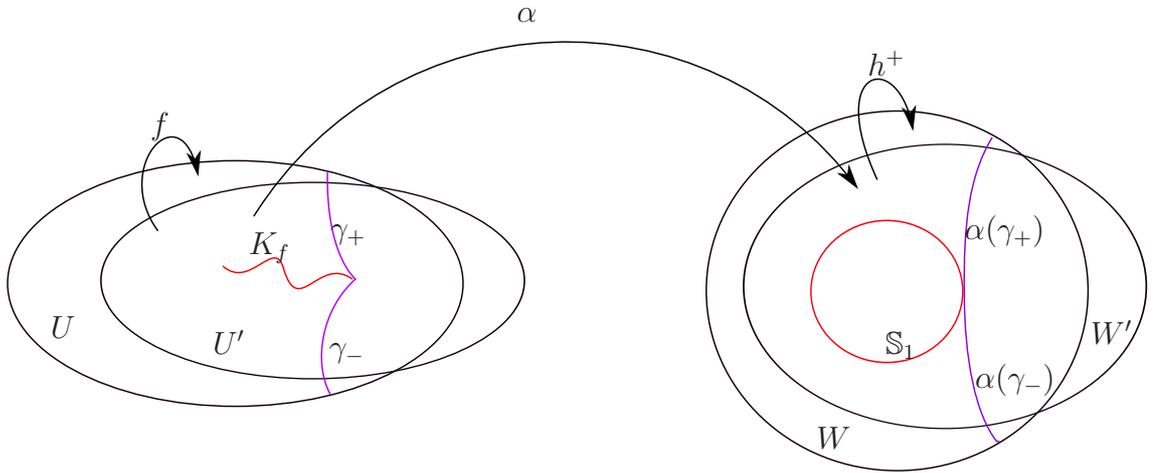}
\caption{\small Construction of an external map
in the case $K_f$ connected.
We set $W'=\alpha(U' \setminus K_f),\,\,\,W=\alpha(U \setminus K_f)$ and $h^+:W' \rightarrow W$.}
\label{ABB}
\end{figure}
Let $\tau(z)= 1/\bar{z}$ denote the reflection with respect to the unit circle, and define $W_-=\tau(W)$,
$W'_-=\tau(W')$, $\widetilde{W}=W \cup \S^1 \cup W_-$ and $\widetilde{W'}=W' \cup \S^1 \cup W'_-$.
Applying the strong reflection principle with respect to $\S^1$ we can extend analytically the map  $h^+: W'
\rightarrow W$ to $h: \widetilde{W'} \rightarrow \widetilde{W}$.
Let $h_f$ be the restriction of $h$ to the unit circle,
then the map $h_f: \S^1 \rightarrow \S^1$ is an \textit{external map} of $f$.
An external map of a parabolic-like map is defined up to a real-analytic diffeomorphism.

\subsubsection*{The general case}\label{gc}
Let ($f,U',U,\gamma$) be a parabolic-like map of degree $d$. To deal with the case where the filled Julia set is not
connected, we will lean on the similar construction in the polynomial-like case. We construct annular Riemann
surfaces $T$ and $T'$ that will play the role of $U' \setminus K_f$ and $U \setminus
K_f$ respectively, and an analytic map $F: T \rightarrow T'$ that will play the role of $f$.

Let $V\approx \D$ be a full relatively compact connected subset of $ U $ containing
$\overline \Omega'$, the critical values of $f$ and such that $(f, f^{-1}(V), V ,\gamma)$ (after riscaling $\gamma$) is a
parabolic-like restriction of ($f, U',U,\gamma$). Call $L=f^{-1}(\overline V)\cap \overline \Omega'$ 
and $M=f^{-1}(\overline V)\cap \Delta'$. Define $X'_0=(U\cup U')\setminus L$, $U_0=U \setminus \overline{V}$, 
$A_0=U \cap U' \setminus L$, $X_0=U\setminus L$, $A_0'=U' \setminus L$ and $A''_0= U'\setminus f^{-1}(\overline V)$.
Note that $X_0$ is an annular domain.

Let $\rho_0: X_1 \rightarrow X_0$ be a degree $d$ covering map for some Riemann surface $X_1$, and define
$V_1=\rho_0^{-1}(V\setminus L)$. Define $X_1''=X_1 \setminus \overline{V_1}$. The map $f: A_0'' \rightarrow U_0$ is
proper holomorphic of
degree
d, and $\rho_0: X_1'' \rightarrow U_0 $ is a proper holomorphic map of degree d. Therefore we
can choose $\pi_0: A_0'' \rightarrow X_1''$, a lift of $f: A_0'' \rightarrow U_0$
to $\rho_0: X_1'' \rightarrow U_0 $, and $\pi_0$ is an isomorphism. 
The subset $\Delta$ has $d$ preimages under the map $\rho_0$. Let us call $\Delta_1$ the preimage of  $\Delta$ under
$\rho_0$ such that $\Delta_1 \cap \pi_0(A_0''\cap \Delta') \neq \varnothing$. 
Since $f: \Delta' \rightarrow \Delta$ is an isomorphism, we can extend the map $\pi_0$ to $\Delta'$.
Let us call $B_1'=X_1''\cup \Delta_1$. Since $\pi_0 (\Delta' \setminus A_0'') \cap X_1''=\varnothing$, the extension
$\pi_0:A_0' \rightarrow B_1'$ is an isomorphism (see Fig \ref{gencase1}).
\begin{figure}[hbt!]
\centering
\psfrag{U'}{$U'$}
\psfrag{U}{$U$}
\psfrag{A0}{$A'_0$}
\psfrag{X0}{$U_0$}
\psfrag{X1}{$B_1'$}
\psfrag{L}{$L$}
\psfrag{f}{$f$}
\psfrag{r0}{$\rho_0$}
\psfrag{p}{$\pi_0$}
\psfrag{M}{$M$}
\psfrag{V}{$V$}
\psfrag{V'}{$V_1$}
\psfrag{D1}{$\Delta_1$}
\includegraphics[width= 13cm]{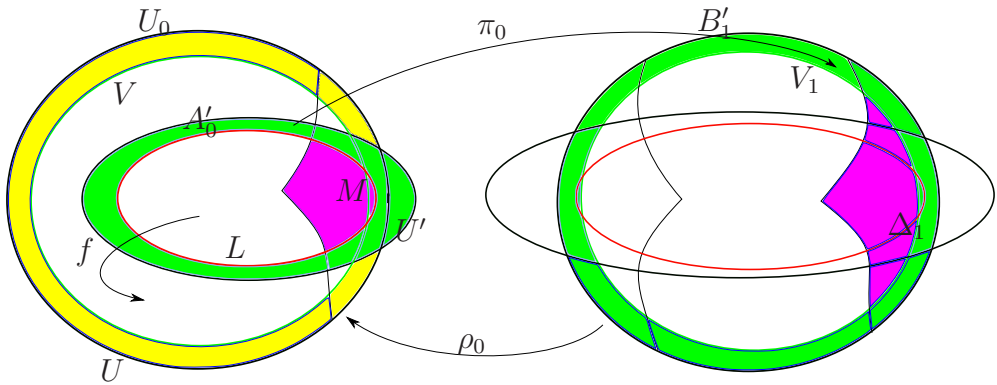}
\caption{\small On the left: in yellow $U_0=U\setminus \overline{V}$, in green plus purple $A_0'=U' \setminus L$. On the
right: in
green plus purple $B_1'=X_1''\cup \Delta_1$. The map $\pi_0: A_0' \rightarrow B_1'$
is an isomorphism.}
\label{gencase1}
\end{figure}
Let us call $B_1=\pi_0(A_0)$.
Define $A_1'=\rho_0^{-1}(A_0)$ and  $f_1= \pi_0 \circ \rho_0 : A_1' \rightarrow B_1$. The map $f_1$ is proper,
holomorphic and of
degree $d$ (see Fig.\ref{gencase2}). Indeed
$\rho_0: A_1' \rightarrow A_0$ is a degree $d$ covering by definition and $\pi_0: A_0 \rightarrow B_1$ is
an isomorphism because it is a restriction of an isomorphism.
Define $X_1'=X_1\setminus \pi_0(A_0' \setminus A_0)$, then $B_1 \subset X_1'$.
\begin{figure}[hbt!]
\centering
\psfrag{A0}{$A_0$}
\psfrag{A1}{$A'_1$}
\psfrag{X0}{$U_0$}
\psfrag{X1}{$B_1$}
\psfrag{L}{$L$}
\psfrag{f}{$f$}
\psfrag{f1}{$f_1:A'_1 \rightarrow B_1$}
\psfrag{r0}{$\rho_0$}
\psfrag{p}{$\pi_0$}
\psfrag{M}{$M$}
\psfrag{V}{$V$}
\psfrag{V'}{$V'$}
\psfrag{D1}{$\Delta_1$}
\includegraphics[width= 13cm]{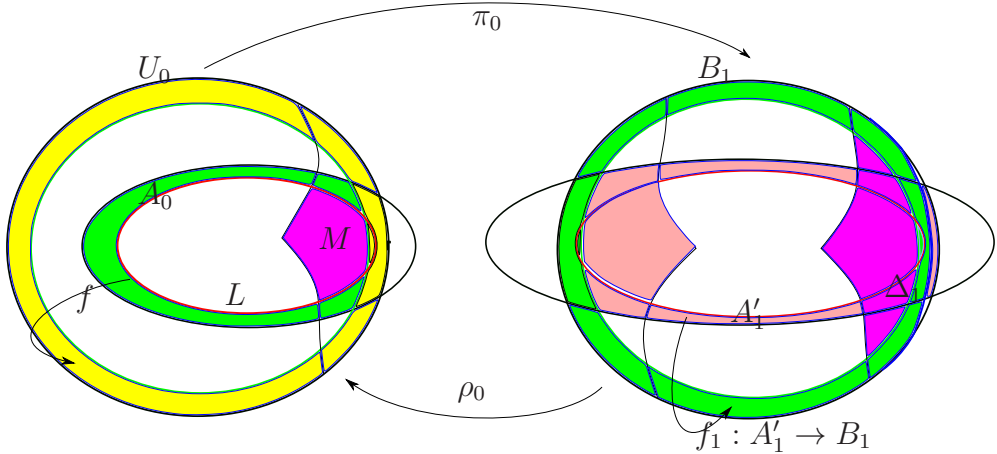}
\caption{\small The map $f_1= \pi_0 \circ \rho_0 : A'_1 \rightarrow B_1$ is proper holomorphic of degree
$d$.}
\label{gencase2}
\end{figure}

Let $\rho_1: X_2 \rightarrow X_1'$ be a degree $d$ covering map for some Riemann surface $X_2$, and
call $B'_2=\rho_1^{-1}(B_1)$. Define
 $\pi_1: A_1' \rightarrow B'_2$ as a lift of $f_1$ to $\rho_1$. Then $\pi_1$ is an isomorphism, since $f_1: A'_1
\rightarrow B_1$ is a degree 
$d$ covering and $\rho_1:B'_2 \rightarrow B_1$ is a degree $d$ covering as well. 
Define $A_1=A_1' \cap X_1'$, and $B_2=\pi_1(A_1)$.
Define $A_2'=\rho_1^{-1}(A_1)$  and  $f_2= \pi_1 \circ \rho_1 : A_2' \rightarrow B_2$. The map $f_2$ is proper,
holomorphic and
of degree $d$, indeed $\rho_1 : A_2' \rightarrow A_1$ is a degree $d$ covering and $\pi_1:A_1\rightarrow B_2$ is an
isomorphism.
Define $X_2'=X_2 \setminus \pi_1(A'_1\setminus A_1)$, then $B_2 \subset X_2'$.
\begin{figure}[hbt!]
\centering

\psfrag{A1}{$A_1$}
\psfrag{X0}{$B_1$}
\psfrag{X1}{$X_2'$}
\psfrag{f1}{$f_1$}
\psfrag{f}{$f_1$}
\psfrag{r0}{$\rho_1$}
\psfrag{p}{$\pi_1$}

\includegraphics[width= 13cm]{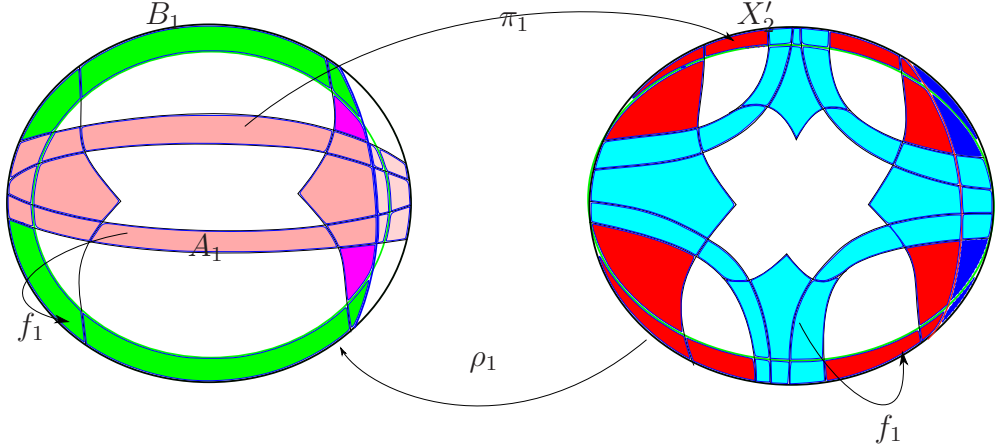}
\caption{\small The map $\pi_1: A_1' \rightarrow B'_2$ is a lift of $f_1$ to $\rho_1$, and it is an
isomorphism.}
\label{gencase3}
\end{figure}

Define recursively $\rho_{n-1}: X_n \rightarrow X'_{n-1}$ for $n>1$ as a holomorphic degree $d$ covering for some
Riemann
surface $X_n$ and call $B'_n= \rho_{n-1}^{-1}(B_{n-1})$. Define recursively $\pi_{n-1}: A_{n-1}' \rightarrow B'_n
\subset X_n$ as a lift of
$f_{n-1}$ to $\rho_{n-1}$. Then $\pi_{n-1}$ is an isomorphism.
Define $A_{n-1}=A_{n-1}' \cap X_{n-1}'$, and $B_n=\pi_{n-1}(A_{n-1})$.
Define $A_n'=\rho_{n-1}^{-1}(A_{n-1})$  and  $f_n= \pi_{n-1} \circ \rho_{n-1} : A_n' \rightarrow B_n$. Then all the
$f_n$ are proper holomorphic maps
of degree $d$, indeed $\rho_{n-1} : A_n' \rightarrow A_{n-1}$ are degree $d$ coverings and
$\pi_{n-1}:A_{n-1}\rightarrow B_n$ are isomorphisms.
Define $X_n'=X_n \setminus \pi_{n-1}(A_{n-1}'\setminus A_{n-1})$, then $B_n \subset X_n'$.

 We define $X'=\coprod_{n\geq0} X'_n$ and $X=\coprod_{n\geq1} X_n$ (disjoint
union). Let $T'$ be the quotient of $X'$
by the equivalence relation identifying $x \in A'_n$ with $x'=\pi_n(x) \in X_{n+1}$, and $T$ be the quotient of $X$ by
the same equivalence relation. Then $T'$ is an annulus, since it is constructed by identifying at each level an
inner annulus $A_i \subset X_i'$ with an outer annulus $B_{i+1}\subset X_{i+1}'$ in the next level.
Similarly $T$ is an annulus, since it is constructed by identifying at each level an
inner annulus $A_i' \subset X_i$ with an outer annulus $B'_{i+1}\subset X_{i+1}$ in the next level. 
Hence (since $\forall i>1, \,\, X'_i \subset X_i$) $T \cup T'=T \cup X_0'/
\sim$ is an annulus, since $X_0'$ is an annulus and $\pi_0$ identifies an inner annulus of $X_0'$ (which is $A_0'$) with
an outer annulus of $X_1$ (which is $B_1'$), and $T$ is an annulus. 
The covering maps $\rho_n$ induce
a degree $d$
holomorphic covering map $F:T \rightarrow T'$. Indeed, $F$ is well defined, since at each level $f_n=
\pi_{n-1} \circ \rho_{n-1}$ by definition and $\pi_n$ is a lift of
$f_{n}$ to $\rho_{n}$. Therefore $\rho_n \circ \pi_n = f_n=\pi_{n-1} \circ \rho_{n-1}$, and the following diagram
commutes
\begin{equation}\label{g}\begin{CD}
A_n'  @>\pi_n>> B'_{n+1} \\
@VV \rho_{n-1} V @VV\rho_n V\\
A_n @>\pi_{n-1}>> B_n
\end{CD}\end{equation}
Finally, the map $F$ is proper of degree $d$ since by definition $F_{|X_n}=\rho_{n-1}:X_n \rightarrow X_{n-1}'$ is a
proper map (and
 $F_{|X_1}=\rho_0: X_1 \rightarrow X_0'$ is proper onto its range, which is $X_0$).

Now, let us construct an external map for $f$. Let $m>0$ be the modulus of the annulus $T\cup T'$. Let
$A
\subseteq \C$ be any annulus with inner boundary $\S^1$ and modulus $m$. Then there exists an isomorphism 
$$\alpha : T\cup T' \longrightarrow A\,\,\,\,\,\,\,\,\,\,\,\,\, (2)$$
with $|\alpha(z)| \rightarrow 1$ when $z \rightarrow L$ and $\alpha(z) \rightarrow 1$ when $z \rightarrow z_0$
within $\Delta/\sim $ (where $\Delta/\sim = \{z \, | \, \exists \, n \, :\, \pi_0^{-1} \circ ... \circ \pi_{n-1}^{-1}
\circ \pi_n^{-1}(z) \, \in \, \Delta \cup \Delta' \}$). 
Then we just have to repeat the construction done for the case $K_f$ connected, with $T$ and $T'$ playing the role of $U' \setminus K_f$ and $U \setminus
K_f$ respectively, and $F$ playing the role of $f$.

\subsection{External equivalence}
\begin{defi}
Two parabolic-like maps ($f,U',U,\gamma_{f}$) and ($g,V',V,\gamma_{g}$) are \textit{externally equivalent}
if their external maps are conjugate by a real-analytic diffeomorphism, \textit{i.e.} if 
their external maps belong to the same external class.
\end{defi}

Let ($f,U',U,\gamma_{f}$) and ($g,V',V,\gamma_{g}$) be two parabolic-like mappings with connected Julia sets.
By the construction of an external map we gave (see Section \ref{em}),
it is easy to see that ($f,U',U,\gamma_{f}$) and ($g,V',V,\gamma_{g}$) are externally equivalent if
and only if there exist parabolic-like restrictions ($f,A',A,\gamma_{f}$) and ($g,B',B,\gamma_{g}$),
and a biholomorphic map
$$\psi : (A \cup A') \setminus K_f \rightarrow (B \cup B') \setminus K_g$$
such that $\psi(\gamma_{\pm f})=\gamma_{\pm g}$ and $\psi \circ f = g \circ \psi$ on $A'\setminus K_f$.
We call $\psi$ an \textit{external equivalence} between $f$ and $g$.

The following Lemma shows that the situation is analogous also in the case where the Julia sets are not connected.
\begin{lemma}\label{lem}
Let $(f_i,U'_i,U_i,\gamma_i)$, $i=1,2$ be two parabolic-like mappings
with disconnected Julia sets. 
Let $W_i \approx \D$ be a full relatively compact connected subset of $U_i$ containing $ \overline{\Omega_i'} $ and the
critical
values of $f_i$, and such that $(f_i,\,f_i^{-1}(W_i),\, W_i,\,\gamma_i)$ is a parabolic-like restriction of
$(f_i,\,U'_i,\,U_i,\,
\gamma_i)$. Define $L_i:=f_i^{-1}(\overline W_i) \cap \overline \Omega_i'$.
Suppose $$\overline{\varphi} : (U_1\cup U_1') \setminus L_1 \rightarrow (U_2 \cup U_2') \setminus L_2$$ is a biholomorphic
map such that $\overline{\varphi} \circ f_1= f_2 \circ \overline{\varphi}$ on $U_1' \setminus L_1$.
Then $(f_1,U'_1,U_1,\gamma_1)$ and $(f_2,U'_2,U_2,\gamma_2)$
are externally equivalent, and we say that
$\overline{\varphi}$ 
is an external equivalence between them.
\end{lemma}
\begin{proof}
Let ($X_{ni},\rho_{(n-1)i},\pi_{(n-1)i},f_{n i})_{n \geq 1,\,i=1,2}$ be as
in the construction of an external map for a parabolic-like map with disconnected Julia set.
Let us set $\varphi_0=\overline{\varphi}$ and define recursively 
$\varphi_n=\rho_{(n-1)2}^{-1} \circ \varphi_{n-1} \circ \rho_{(n-1)1} : X_{n1} \rightarrow X_{n2}$.
Then the following diagram commutes:
\begin{equation}\label{g}\begin{CD}
X_{n1}' \subset X_{n1} @>\varphi_n>> X_{n2}\supset X'_{n2} \\
@VV \rho_{(n-1)1} V @VV\rho_{(n-1)2} V\\
X'_{(n-1) 1} @>\varphi_{n-1}>> X'_{(n-1) 2}
\end{CD}\end{equation}
(for $n=0$, $\rho_{0, i}: X_{1,i} \rightarrow X_{0,i} \subset X_{0,i}'$).
Then every $\varphi_n:X_{n1} \rightarrow X_{n2}$ thus defined is an isomorphism and a conjugacy between $f_{n 1}$ and
$f_{n 2}$, and the following diagram commutes:
 \begin{equation}\label{g}\begin{CD}
X_{n1}\supset A'_{n1}@>f_{n 1}>> B_{n1}\subset X_{n1}' \\
@VV \varphi_n V @VV\varphi_n V\\
X_{n2}\supset A'_{n2} @>f_{n 2}>> B_{n2}\subset X_{n2}'
\end{CD}\end{equation}
Thus the family of isomorphisms $\varphi_n$ induces an isomorphism
$\Phi: T_1 \cup T'_1 \rightarrow T_2 \cup T'_2$ compatible with dynamics, and so the external maps of $f_1$ and $f_2$
are real-analytically
conjugate. 
\end{proof}
\subsubsection{External map for the members of the family $Per_1(1)$}
The filled Julia set $K_P$ of a polynomial $P: \widehat\C \rightarrow \widehat\C$ is defined as the complement of the
basin of attraction of infinity, which is a completely invariant Fatou component.
For a degree $d$ rational map $R: \widehat\C \rightarrow \widehat\C$
with a completely invariant Fatou component $\Lambda$ we may define the filled Julia set as 
$$K_R= \widehat \C \setminus \Lambda.$$ 
Note that a degree $d$ map can have up to $2$ completely invariant Fatou components $\Lambda_1,\,\Lambda_2$ (since a
degree $d$ map defined on the Riemann sphere has $2d-2$ critical points, and a completely invariant Fatou component contains
at least $d-1$ critical points).
In the case $R$ has precisely $1$ completely invariant Fatou component $\Lambda$, the filled Julia set $K_R= \widehat \C
\setminus \Lambda$ is well defined. In the case $R$
has $2$ such components $\Lambda_1,\,\Lambda_2$, there are $2$ possibilities for the
filled Julia set, hence we need to make a choice. After choosing a completely invariant component $\Lambda_*$, the
filled Julia set $K_R=\widehat \C \setminus \Lambda_*$ is well defined. 

Every member of the family $Per_1(1)$ has a parabolic fixed point at $\infty$ with multiplier $1$, and the basin of attraction of the parabolic
fixed point is a completely invariant Fatou component. For all the members of the family $Per_1(1)$ with $A \neq 0$
the parabolic multiplicity of the parabolic fixed point is $1$, hence all these maps have precisely one completely invariant Fatou
component $\Lambda$. Thus for all the members of the family $Per_1(1)$ with $A \neq 0$ the filled Julia set
$K_{P_A}= \widehat \C \setminus \Lambda$ is well defined.
On the other hand, since for the map $P_0(z)= z+ 1/z$ the parabolic multiplicity of $\infty$ is $2$,
this map has 2 completely invariant Fatou components, namely $\H_r$ and $\H_l$. Since $P_0(z)=z+1/z$ is conformally
conjugate to the map $h_2(z)=\frac{3z^2+1}{3+z^2}$ under the map $\varphi(z)=\frac{z+1}{z-1}$,
for consistency with the Example 1 in Section \ref{ex} we consider
$K_{P_0}=\overline{\H}_l= \varphi(\overline{\D})$.

Let $f: \widehat{\C} \rightarrow \widehat{\C}$ be a rational map of degree $d$. The
map $f$ has a
\textit{parabolic-like restriction} if there exist open connected sets $U$ and $U'$ and a dividing
arc $\gamma$ such that ($f,U',U,\gamma$) is a parabolic-like map of some degree $d' \leq d$.
A parabolic-like restriction of a member $P_A$ of the family $Per_1(1)$ has degree $2$,
hence the filled Julia set $K_{P_A}$ defined as above coincides with the filled Julia set of the parabolic-like
restriction of $P_A$. Therefore, we consider as external class of $P_A$ the external class of its parabolic-like
restriction.

\begin{prop}\label{extmap}
For every $A \in \C$ the \textit{external class of $P_A$ } is given by the class of
$h_2(z)=\frac{z^2+\frac{1}{3}}{1+\frac{z^2}{3}}$.
\end{prop}

\begin{proof}
Since the maps $P_0(z)=z+1/z$ and $h_2(z)=\frac{3z^2+1}{3+z^2}$ are conformally
conjugate, in order to prove that $h_2$ is an external map of $P_A$, it is sufficient to prove that $P_0$ 
is externally equivalent to $P_A$, for $A \in \C$.
Let $\Xi^0$ be an attracting petal of $P_0$ containing the critical value $z=2$, and let $\phi_0: \Xi^0 \rightarrow
\mathbb{H}_r$ be the incoming  Fatou coordinates of $P_0$ normalized by $\phi_0(2)=1$.
Replacing $A$ by $-A$ if necessary, we can assume that $z=1$ is the first critical point attracted by
$\infty$. Let
$\Xi^A$ be an attracting petal of $P_A$ and let $\phi_A: \Xi^A \rightarrow \mathbb{H}_r$ be the incoming  Fatou
coordinates of $P_A$ with $\phi_A(2+A)=1$.

Let us contruct an external equivalence between $P_0$ and $P_A$ first in the case $K_{P_A}$ is connected.
The map $\eta:=\phi_A^{-1} \circ \phi_0: \Xi^0 \rightarrow \Xi^A$ is a conformal
conjugacy between $P_0$ and $P_A$ on $\Xi^0$. Defining $\Xi^0_{-n},\, n>0$ as the connected component of 
$P_0^{-n}(\Xi^0)$ containing $\Xi^0$, and $\Xi^A_{-n},\, n>0$ as the connected component of
$P_A^{-n}(\Xi^A)$ containing $\Xi^A$, we can lift the map $\eta$ to $\eta_n: \Xi^0_{-n} \rightarrow
\Xi^A_{-n}$.
Since $K_{P_A}$ is connected by iterated lifting of $\eta$ we obtain a conformal conjugacy
$\overline{\eta}:\widehat{\C} \setminus
K_{P_0} \rightarrow \widehat{\C} \setminus K_{P_A}$ between $P_0$ and $P_A$.

In the case $K_{P_A}$ is not connected the map $\eta$ is a conformal
conjugacy between $P_0$ and $P_A$ on the region delimited by the Fatou equipotential passing through $z=1$.
\begin{figure}[hbt!]
\centering
\psfrag{D0'}{$D_0'$}
\psfrag{D'}{$D'$}
\psfrag{D}{$D$}
\psfrag{ga+}{$\gamma_{+_A}$}
\psfrag{ga-}{$\gamma_{-_A}$}
\psfrag{g0+}{$\gamma_{+_0}$}
\psfrag{g0-}{$\gamma_{-_0}$}
\psfrag{fg+}{$\tilde{\gamma}_+$}
\psfrag{fg-}{$\tilde{\gamma}_-$}
\psfrag{delta0}{$\Delta'_0$}
\psfrag{deltaa}{$\Delta'_A$}
\psfrag{U0'}{\small$(U_0')^c$}
\psfrag{U}{$\D(z_0,r')$}
\psfrag{D0}{$D_0$}
\psfrag{U0}{\small $U_0^c$}
\psfrag{DA'}{$D_A'$}
\psfrag{UA'}{\small$(U_A')^c$}
\psfrag{DA}{$D_A$}
\psfrag{UA}{\small$U_A^c$}
\psfrag{U'}{$T^{-1}(\D(z_0,r'))$}
\psfrag{2}{$2$}
\psfrag{1}{$1$}
\psfrag{2A}{\small $2+A$} 
\psfrag{-2A}{\small$-2+A$}
\psfrag{0}{$0$}
\psfrag{e}{$\eta$}
\psfrag{fi0}{$\phi_0$}
\psfrag{fiA}{$\phi_A$}
\psfrag{fi2A}{$\phi_A(-2+A)$}
\includegraphics[width= 15cm]{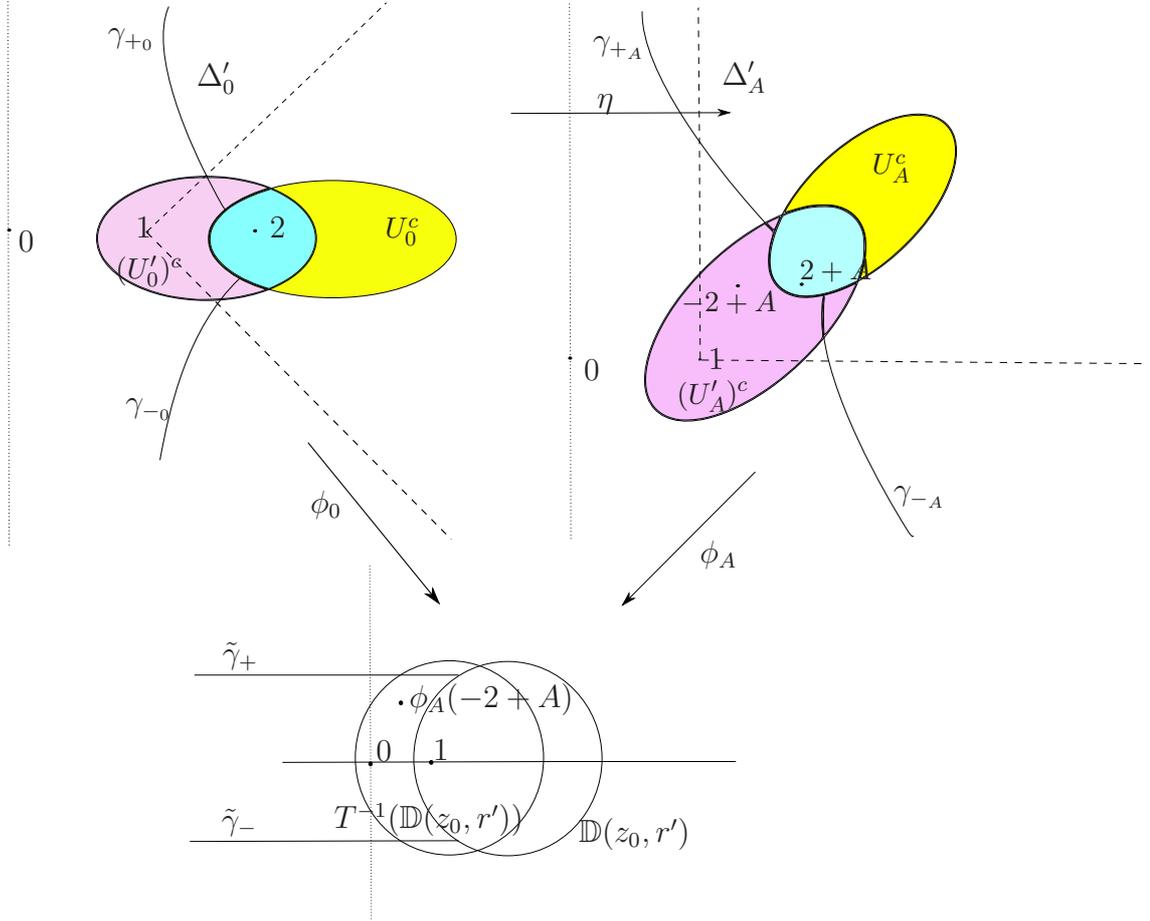}
\caption{\small The construction of parabolic-like restrictions of $P_0$ and $P_A$. In the picture we are assuming
the critical value $z=-2+A$ in
$\Omega_A \setminus \Omega'_A$. In this case the critical value $z=-2+A$ belongs to the attracting petal $\Xi_A$.}
\label{exteqnuovo1}
\end{figure}
We are now going to construct parabolic-like restrictions ($P_0,U'_0, U_0, \gamma_0$) and ($P_A,U'_A,
U_A,\gamma_A$) of the maps $P_0$ and $P_A$ respectively and extend the map 
$\eta$ to an external equivalence between them.
The critical point $z=1$ is the first attracted by infinity for both the maps $P_0$ and $P_A$, so it cannot
belong to the domains $U'_0,\,U'_A$ of their parabolic-like restrictions but it may belong to the
codomains $U_0,\,U_A$, while the critical point $z=-1$ belongs to $\Omega'_0$ and $\Omega'_A$.
Let us denote by $\widehat{\phi_A}$ and $\widehat{\phi_0}$ the Fatou coordinates of $P_A$ and $P_0$ respectively (normalized
by $\widehat{\phi_A}(2+A)=1$ and $\widehat{\phi_0}(2)=1$), extended to
the whole basin of attraction of $\infty$ by iterated lifting. 
The maps $\widehat{\phi_A}$ and $\widehat{\phi_0}$ have univalent inverse branches 
 \begin{displaymath}
      \psi_A: \C \setminus \{z=x+iy | x<0
\wedge y \in [0, Im\widehat{\phi_A}(-2+A) ]\} \rightarrow \widehat \Xi_A
  \end{displaymath}
and $\psi_0: \C \setminus \R_- \rightarrow
\widehat \Xi_0$ respectively, and the map $$\eta = \psi_A \circ \widehat \phi_0 : \psi_0^{-1}(\C \setminus \{z=x+iy | x<0
\wedge y \in [0, Im\widehat{\phi_A}(-2+A) ]\}) \rightarrow \widehat \Xi_A$$
is a biholomorphic extension of $\eta$ conjugating dynamics.
Choose $r> max \{ 1+ Im(\widehat{\phi_A}(A-2)),\,\, 2 \}$ and $z_0,\,\,r<z_0<r+1$ such
that $A-2 \notin  \phi_A^{-1}(\overline{\D(z_0,r)})$. Then for $r<r'<z_0$ with $r'$ sufficiently close to
$r$ we have $A-2
\notin \phi_A^{-1}(\D(z_0,r'))$.
Let $\widetilde{\gamma}_+$, $\widetilde{\gamma}_-$ be horizontal lines, symmetric with respect to the real axis, 
starting at $- \infty$ and landing at $\partial
\D(z_0,r)$, such that the point $\widehat{\phi_A}(A-2)$ is contained in the strip between them (see Fig.
\ref{exteqnuovo1}) and they do not leave the disk $T^{-1}(\D(z_0,r))$ (where $T^{-1}(\D(z_0,r))$ is the disk of radius $r$ and center
$z_1=z_0-1$) after having entered to it.
Define $U_0=(\phi_0^{-1}(\D(z_0,r))^c$, $U'_0=P_0^{-1}(U_0)$, $\gamma_{+_0}=\psi_0(\widetilde{\gamma}_+)$, and
$\gamma_{-_0}=\psi_0(\widetilde{\gamma}_-)$.
In the same way define $U_A=(\phi_A^{-1}(\D(z_0,r))^c$, $U'_A=P_A^{-1}(U_A)$, $
\gamma_{+_A}=\psi_A(\widetilde{\gamma}_+),$ and
$\gamma_{-_A}=\psi_A(\widetilde{\gamma}_-)$.
Then the parabolic-like restriction of $P_0$ we consider is
($P_0,U'_0, U_0, \gamma_{+_0}, \gamma_{-_0}$), and the parabolic-like 
restriction of $P_A$ we consider is
($P_A,U'_A, U_A, \gamma_{+_A}, \gamma_{-_A}$).
Note that, by construction, the map $\eta$ is a conformal conjugacy between $P_0$ and $P_A$ on $\Delta'_0$.
\begin{figure}[hbt!]
\centering
 \psfrag{D0'}{$D_0'$}
 \psfrag{D'}{$D'$}
 \psfrag{D}{$D$}
 \psfrag{ga+}{$\gamma_{+_A}$}
 \psfrag{ga-}{$\gamma_{-_A}$}
 \psfrag{g0+}{$\gamma_{+_0}$}
 \psfrag{g0-}{$\gamma_{-_0}$}
 \psfrag{fg+}{$\tilde{\gamma}_+$}
\psfrag{fg-}{$\tilde{\gamma}_-$}
\psfrag{delta0}{$\Delta'_0$}
 \psfrag{deltaa}{$\Delta'_A$}
 \psfrag{U0'}{\small $(U_0')^c$}
 \psfrag{U}{$\D(z_0,r)$}
 \psfrag{D0}{$D_0$}
 \psfrag{U0}{\small $U_0^c$}
 \psfrag{DA'}{$D_A'$}
 \psfrag{UA'}{\small $(U_A')^c$}
 \psfrag{DA}{$D_A$}
 \psfrag{UA}{\small $U_A^c$}
 \psfrag{U'}{$T^{-1}(\D(z_0,r))$}
 \psfrag{2}{$2$}
 \psfrag{1}{$1$}
 \psfrag{2A}{\small $2+A$}
 \psfrag{-2A}{\small $-2+A$}
 \psfrag{0}{$0$}
\psfrag{e}{$\eta$}
 \psfrag{fi0}{$\phi_0$}
 \psfrag{fiA}{$\phi_A$}
 \psfrag{fi2A}{$\phi_A(-2+A)$}
\includegraphics[width= 13cm]{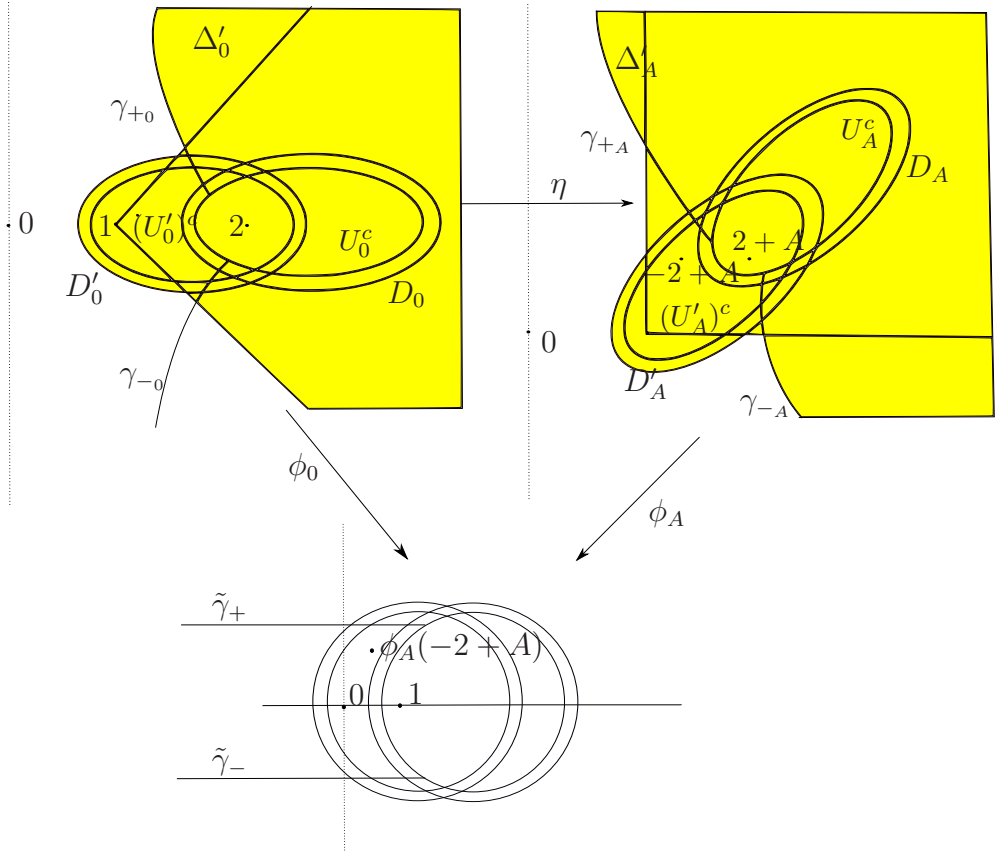}
\caption{\small The construction of the external equivalence $\eta$ between the parabolic-like restriction of $P_0$ and
 the parabolic-like restriction of $P_A$. For $r<r'$, $\D(z_0,r) \subset \subset \D(z_0,r')$ and $T^{-1}(\D(z_0,r)) \subset \subset T^{-1}(\D(z_0,r'))$.
 In the picture we are assuming the critical value $z=-2+A$ in
$\Omega_A \setminus \Omega'_A$.}
\label{exteqnuovo2}
\end{figure}
In order to obtain an external equivalence we need $\eta$ to be defined on a fundamental annulus.
Define $D_0=\phi_0^{-1}(\D(z_0,\,r'))$, $D'_0=P_0^{-1}(D_0)$, $D_A=\phi_A^{-1}(\D(z_0,\,r'))$,
and $D'_A=P_A^{-1}(D_A)$ (see Fig. \ref{exteqnuovo2}).
Since $D_0$ and $D_A$ belong to the regions delimited by the Fatou equipotential passing through $z=1$, the restriction
$\eta: D_0 \rightarrow D_A $  is a
holomorphic conjugacy between $P_0$ and $P_A$. Since $-2 \notin D_0$ and $-2+A \notin D_A$, the restrictions
$P_0:D'_0 \setminus \{1\} \rightarrow D_0 \setminus \{2\}$ and $P_A:D'_A \setminus \{1\}\rightarrow
D_A \setminus \{2+A\}$ are degree $2$ coverings.
Hence we can lift the map $\eta$ to $\eta: D'_0 \setminus \{1\} \rightarrow D'_A \setminus \{1\}$.
Finally, we obtain a biholomorphic map $\eta: D'_0 \cup \Delta'_0 \rightarrow D'_A
\cup \Delta'_A$ which conjugates dynamics.

Define $V_0=(\overline {D_0})^c$ and $V_A=(\overline {D_A})^c$, and consequently
$L=\overline{\Omega_0'\setminus D'_0}$ and $M=\overline{\Omega_A'\setminus D'_A}$.
The sets $V_0$ and $V_A$ are compactly contained in $U_0$ and $U_A$ respectively,
containing respectively $\overline{\Omega_0'}$ (which contains the critical value $-2$), and $\overline{\Omega_A'}$ and the critical
value $-2+A$, and such that $P_0: (\overline D'_0)^c \rightarrow (\overline D_0)^c$ and $P_A: (\overline D'_A)^c \rightarrow (\overline D_A)^c$ are parabolic-like
restrictions of ($P_0,U_0, U'_0, \gamma_0$) and ($P_A,U_A, U'_A, \gamma_A$)
respectively, and the map $\eta: (U_0 \cup
U_0')\setminus L \rightarrow (U_A \cup U_A')\setminus M$ is a biholomorphic conjugacy between $P_0$ and $P_A$.
Therefore the result follows by Lemma \ref{lem}.
\end{proof}

\subsection{Properties of external maps}\label{propex}

Let ($f,U',U,\gamma$) be a parabolic-like map of degree $d$, and let $h_f$ be a representative of its external class. 
The map $h_f: \S^1 \rightarrow \S^1$ is by construction real analytic,
symmetric with respect to the unit
circle, and it has a parabolic fixed point $z_1$ of multiplier $1$ and even parabolic multiplicity $2n$ (where
$n$ is
the
number of
petals 
of $z_0$ outside $K_f$). Let $\alpha$ be an isomorphism which defines $h_f$. Hence $h_f$ inherits via $\alpha$ dividing arcs $\gamma_{h_f +}:=
\alpha(\gamma_+ \setminus \{z_0\})\cup \{z_1\}$ and
$\gamma_{h_f-}:=\alpha(\gamma_-\setminus\{z_0\})\cup \{z_1\}$, which
divide $W_f'\setminus \D$ and $ W_f \setminus \D$ into $\Omega'_W, \Delta'_W$ and $\Omega_W,
\Delta_W$
respectively, such that $h_f:\Delta_W' \rightarrow \Delta_W$ is an isomorphism and
$\Delta_W'$ contains at least one attracting fixed petal of $z_1$. Note that $\Omega'_W$ is not compactly
contained in $W$ (since they share the inner boundary), and that $\gamma_{h_f
+}$ and $\gamma_{h_f -}$
form a positive angle (since there is at least one attracting fixed petal of
$z_1$ in $\Delta_W$. We prove in \cite{LL} that this angle is $\pi$).
Moreover, we prove in \cite{LL} (Theorem 2.3.3) that
there exists $\hat{h} \in [h_f]$
such that for all $z \in \S^1$, $|\hat h(z)| \geq 1$, and the equality holds only at the parabolic fixed point.

\section{Parabolic external maps}
So far we have considered external maps only in
relation
to parabolic-like maps (and members of the family $Per_1(1)$).
We now want to separate these two concepts, and then consider external maps
as maps of the unit circle to itself with some specific properties, without refering to a particular parabolic-like
map.
In order to do so we need to give an abstract definition of external map, which endows it with all the properties it
would have, if it would have been constructed from a parabolic-like map.

\begin{defi}\textbf{(Parabolic external map)\,\,\,}
Let $h:\S^1 \rightarrow \S^1$ be a degree $d$ orientation preserving real-analytic and metrically expanding (i.e. $|h'(z)| \geq 1$) 
map.
We say that
$h$ is a \textit{parabolic external map,} if there exists a
unique $z=z_*$ such that $h(z_*)=z_*$ and $h'(z_*)=1$, and $|h'(z)|>1$ for all $z \neq z_*$.
\end{defi}
The multiplicity of $z_*$ as parabolic fixed point of $h$ is even (since the map $h$
is symmmetric with respect
to the unit circle). As
$h$ is metrically expanding, the repelling petals of $z_*$ intersect the unit circle.
Let $h: W' \rightarrow W$ be an
extension
which is a degree $d$ covering (where $W=\{ z : 1/(1+\epsilon) < |z| < 1 + \epsilon \}$ for an $\epsilon >0$, and
$W'=h^{-1}(W)$). 
We define a \textit{dividing arc for $h$} to be an arc $\widetilde \gamma:[-1,1] \rightarrow \overline{W} \setminus \D$,
forward invariant under $h$, $C^1$ on $[-1,0]$ and $[0,-1]$, residing in the union of the repelling petals which intersect the unit circle and
such that
$$h(\widetilde\gamma(t))=\widetilde\gamma(dt),\,\,\, \forall -\frac{1}{d} \leq t \leq \frac{1}{d},$$
$$\widetilde\gamma([ \frac{1}{d}, 1)\cup (-1, -\frac{1}{d}]) \subseteq W \setminus W',\,\,\,\,\,\,\widetilde\gamma(\pm 1) \in \partial W.$$
\begin{remark}
 A dividing arc for a parabolic external map can be constructed by taking preimages of horizontal
lines by repelling Fatou coordinates 
with axis tangent to the unit circle at the parabolic fixed point. 
\end{remark}
The dividing arc divides $W'\setminus \D$ and $W\setminus \D$ into $\Omega'_W, \Delta'_W$ and $\Omega_W, \Delta_W$
respectively, such that $h:\Delta_W' \rightarrow \Delta_W$ is an isomorphism and
$\Delta_W'$ contains at least an attracting fixed petal of $z_*$. 
We prove in \cite{LL} (see Lemma 2.3.9)
that there exists a range $\widetilde W$ for an extension $h:\widetilde W' \rightarrow \widetilde W$ degree $d$ covering
such that $\Omega_{\widetilde W} \setminus \Omega'_{\widetilde W}$ is a
topological quadrilateral.
Therefore external maps contructed from parabolic-like mappings and parabolic external maps are equivalent concepts.
\begin{remarks}
\begin{itemize}
	\item For clarity of exposition we consider in this paper parabolic-like maps with external map having
exactly one parabolic fixed point (cfr. definition \ref{definitionparlikemap}). This concept naturally generalizes to
maps with external maps having several parabolic fixed points. A
general parabolic-like map has as many pairs of dividing arcs $\gamma_{\pm}$
(which divide $U$ and $U'$ in $\Omega,\Delta_1, \Delta_2$,...,$\Delta_n$ and $\Omega',\Delta_1',
\Delta'_2$,...,$\Delta'_n$ respectively) as the number of parabolic fixed points.  
	\item Moreover this concept generalizes in a similar way to maps with external maps having several parabolic
periodic orbits. An external map for such an object is an orientation preserving real-analytic and metrically expanding
map $h:\S^1 \rightarrow \S^1$ with $h'(z_*)=1$ for every $z_*$ belonging to a parabolic cycle and $|h'(z)|>1$ for all
the other points of the unit circle.
 \end{itemize}
\end{remarks}

\begin{defi}
A \textit{degree $d$ covering extension} $h: W' \rightarrow W$ of a parabolic
external map $h : \S^1 \rightarrow \S^1$ is an extension to some neighborhood $W=\{ z : e^{-\epsilon} < |z| <
e^{\epsilon}\}$ for an $\epsilon >0$, and $W'=h^{-1}(W)$ such that the map $h : W'
\rightarrow
W$ is a degree $d$ covering
and
there exists a dividing arc $\widetilde\gamma$ which divides
$W'\setminus \D,\,\, W\setminus \D$ into $\Omega'_W, \Delta'_W$ and $\Omega_W, \Delta_W$
respectively, such that $\Omega_W \setminus
\Omega'_W$ is a
topological quadrilateral. 
\end{defi}
The concept of \textit{parabolic-like restriction} naturally applies to parabolic external maps 
($h: \hat W' \rightarrow \hat W$ is a parabolic-like restriction of $h: W' \rightarrow W$ if they are both
degree $d$ covering extension of the same parabolic external map and
$\hat W \subseteq W$). Let $\gamma$ be a dividing arc for some parabolic external map.
We say that $\gamma_s:[-1,1] \rightarrow \C \setminus \D$ is \textit{isotopic}
to $\gamma$
if their projections to \'Ecalle cylinders are isotopic and the isotopies are
disjoint from the projections of the unit circle. From the definitions of dividing arc and isotopy of arcs
for parabolic external maps is easy to see that dividing arcs for the same external map are isotopic. 

\begin{prop}\label{arcprop}
 Let $h_i: \S^1 \rightarrow \S^1,\,\,i=1,2$ be parabolic external maps of the same degree $d$, $h_i:
W_i' \rightarrow W_i$ degree $d$ covering extensions, and $\gamma_{i}$ dividing arcs. Then the following statements hold.
\begin{enumerate}
\item Let $\gamma_s:[-1,1] \rightarrow \C \setminus \D$ be a foward invariant arc for $h_1$,
isotopic to $\gamma_1$, and $C^1$ on $[-1,0]$ and $[0,1]$. Then (possibly after rescaling)
$\gamma_s$ is a dividing arc for $h_1$.
\item  Assume $\gamma_{i +}$ and $\gamma_{i -}$ are constructed by taking preimages of the same periodic curves 
by repelling Fatou coordinates $\phi_{i +}$ and $\phi_{i -}$
(with axis tangent to the unit circle at the parabolic fixed point) of $h_i$. Then the map $\phi_{2
}^{-1} \circ \phi_{1}:\gamma_{1} \rightarrow \gamma_{2 }$ defined as:
$$ \phi_{2}^{-1} \circ \phi_{1}(z)=\left\{
\begin{array}{cl}
\phi_{2 +}^{-1} \circ \phi_{1 +} &\mbox{on  } \gamma_{1 +}\\
\phi_{2 -}^{-1} \circ \phi_{1 -} &\mbox{on  } \gamma_{1 -}\\
\end{array}\right.
$$  
is a quasisymmetric conjugacy between $h_{1_|\gamma_1}$ and $h_{2_|\gamma_2}$.

\end{enumerate}

\end{prop}
\begin{proof}
Property $1$ comes from the definition of isotopy in the parabolic external maps setting and a similar argument as in the proof of Lemma \ref{is}.
Let us prove Property $2$. In the following, $i=1,2$.
Let $\Xi_{i +}$ and $\Xi_{i -}$ be repelling petals where $\gamma_{i +}$ and $\gamma_{i -}$ respectively reside,
and let $\phi_{i +}: \Xi_{i +} \rightarrow \H_l$ and $\phi_{i -}: \Xi_{i -} \rightarrow \H_l$ be repelling Fatou coordinates
with axis tangent to the unit circle at the parabolic fixed point $z_i$ of $h_i$. Then there exist
$\gamma_+$ and $\gamma_-$, $1$-periodic curves in $\H_l$ bounded from above and below, such that $\gamma_{i +}= \phi_{i +}^{-1}(\gamma_+)$
and $\gamma_{i -}= \phi_{i -}^{-1}(\gamma_-)$.
The map $\phi_{2}^{-1} \circ \phi_{1}:\gamma_{1} \rightarrow \gamma_{2 }$
is clearly a conjugacy between $h_{1_|\gamma_1}$ and $h_{2_|\gamma_2}$. Let us prove that this map is quasisymmetric.
To fix the notation let us assume the multiplicity of $z_i$ as parabolic fixed point of $h_i$ is $2n_i$.
By an iterative local change of coordinates applied to eliminate lower order terms one by one, we obtain conformal
diffeomorphisms $g_i$ which
conjugate $h_i$ to the map $z \rightarrow z(1+z^{2n_i}\,+\,cz^{4n_i}+ \, O (z^{6n_i}))$ on $\Xi_{i \pm}$. Since the
forward invariant arcs
$\gamma_{i \pm}$ reside in the repelling petals $\Xi_{i \pm}$, it suffices to consider
$h_i(z)=z(1+z^{2n_i}\,+\,cz^{4n_i}+ \,
O
(z^{6n_i}))$.
The map
$I_i(z)=-\frac{1}{2n_iz^{2n_i}}$ conjugates $h_i$ to $h_i^*(z)=\, z \, +\, 1 \,+ \, \hat{c_i} \,
 \frac{1}{z}\, +\, O(\frac{1}{z^2})$. Shishikura proved in \cite{Sh} that Fatou coordinates which conjugate the map
$h_i^*$ to $T(z)=\,z\,+\,1$ on $I_i(\Xi_{i \pm})$ take the form $\Phi_{i \pm}(z) \,=\, z\, - \,\hat{c_i}\,
log(z)\, + \,c_{i \pm} \,+ \,o(1)$.
Therefore $\phi_{i \pm}= \Phi_{i \pm} \circ I_i$, and
we can write:
$$\gamma_{i +}=(\Phi_{i +} \circ I_i)^{-1}(\gamma_+)_{\{i=1,2\}},$$ 
$$\gamma_{i -}=(\Phi_{i -} \circ I_i)^{-1}(\gamma_-)_{\{i=1,2\}} .$$ 
\begin{equation}\label{g}\begin{CD}
\gamma_{i} @>h_i >> \gamma_{i} \\
@VV I_i V @VV I_i V\\
\H_l @>h_i^*>> \H_l\\
@VV \Phi_i V @VV \Phi_i V\\
\H_l @>T>> \H_l
\end{CD}\end{equation}
Call $\gamma^*_{i +}= I_i(\gamma_{i +})$, $\gamma^*_{i -}= -I_i(\gamma_{i -})$ and $\gamma^*_{i}=\gamma^*_{i
+} \cup \infty \cup \gamma^*_{i -}$.
The map $\widehat{I_i}: \gamma_{i} \rightarrow \gamma^*_{i}$:
$$ \widehat{I_i}(z)=\left\{
\begin{array}{cl}
I_i(z) &\mbox{on  } \gamma_{i +}\\
-I_i(z) &\mbox{on  } \gamma_{i -}\\
\end{array}\right.
$$  
is quasisymmetric on a neighborhood of $0$.
Define $\widehat \gamma= \gamma_+ \cup \infty \cup -\gamma_-$, and the map $\widehat \Phi_i: \gamma^*_{i} \rightarrow \widehat\gamma$ as
follows:
$$\widehat \Phi_i(z)=\left\{
\begin{array}{cl}
\Phi_{i +}(z) &\mbox{on  } \gamma^*_{i +}\\
- \Phi_{i -}(-z) &\mbox{on  } \gamma^*_{i -}\\
\end{array}\right.
$$  
The map $\widehat \Phi_i$ is the restriction to $\gamma^*_{i}\setminus \infty$ of a conformal map. Again by Shishikura
\cite{Sh} the maps
$\Phi_{i
+},\,\Phi_{i -}$ have derivatives $\Phi'_{i \pm}=1+o(1)$, hence the map $\widehat \Phi_i: \gamma^*_{i} \rightarrow \widehat\gamma$
is a diffeomorphism (one may take $1/x$ as a chart). 
The map $\widehat \Phi_i \circ \widehat{I_i}: \gamma_i \rightarrow \widehat\gamma$ conjugates the map $h_i$ to the map $T_+(z)=\,z\,+\,1$ on $\gamma_{i +}$, and to
the map $T_-(z)=\,z\,-\,1$ on $\gamma_{i -}$. Hence $\phi_2^{-1}\circ\phi_{1}=(\widehat \Phi_2 \circ \widehat{I_2})^{-1} \circ
(\widehat \Phi_1 \circ \widehat{I_1}): \gamma_{1} \rightarrow \gamma_2$. The map $\widehat \Phi_2^{-1}$ is
a diffeomorphism because it has the same analytic expression as $\widehat \Phi_2$, and therefore the map 
$\widehat\Phi_2^{-1}\circ \widehat\Phi_1$ is a diffeomorphism. Since the map $\widehat{I_i}$ is quasisymmetric on a
neighborhood of $0$, the inverses are quasisymmetric on a neighborhood of $\infty$. Hence the composition
$\phi_2^{-1}\circ\phi_{1}=\widehat{I_2}^{-1} \circ \widehat\Phi_2^{-1}\circ \widehat\Phi_1 \circ\widehat{I_1} : \gamma_{1} \rightarrow
\gamma_2$ is quasisymmetric.
\end{proof}

\section{The Straightening Theorem}

\begin{defi}
Let ($f,U',U,\gamma_{f}$) and ($g,V',V,\gamma_{g}$) be two parabolic-like mappings.
We say that $f$ and $g$ are \textit{holomorphically equivalent} if there exist parabolic-like restrictions 
($f,A',A,\gamma_{f}$) and ($g,B',B,\gamma_{g}$), and a biholomorphic map $\varphi :(A \cup A')
\rightarrow (B \cup B')$
such
that
$\varphi(\gamma_{\pm f})=\gamma_{\pm g}$ and
$$\varphi(f(z))=g(\varphi(z))\,\,\,\,\mbox{ on } A'$$
\end{defi}

\begin{prop}\label{stex}
 A degree $2$ parabolic-like map is holomorphically conjugate to a member of the family $Per_1(1)$
 if and only if its external class is given by the class of $h_2$. 
\end{prop}

\begin{proof}
By Proposition \ref{extmap}, the external class of every member of the family $Per_1(1)$ is given by the 
class of $h_2$, hence a parabolic-like map holomorphically conjugate to a member of the family $Per_1(1)$
has external map in the class of $h_2$.
Let us prove that a degree $2$ parabolic-like map $g: V' \rightarrow V$
with external map $h_2$ is holomorphically conjugate to a member of the family $Per_1(1)$.
Let $\overline{\psi}$ be an external equivalence between
the maps $g$ and $h_2$.
Let $S$ be the Riemann surface obtained by gluing 
$ V \cup V'$ and $\widehat{\C} \setminus \overline{\D}$, by the equivalence relation identifying $z$ to
$\overline{\psi}(z)$, i.e.
$$S= (V \cup V') \coprod (\widehat{\C} \setminus \overline{\D}) / z \sim \overline{\psi}(z).$$ 
By the Uniformization Theorem, $S$ is isomorphic to the Riemann sphere.
Consider the map 
$$ \widetilde{g}(z)=\left\{
\begin{array}{cl}
g &\mbox{on  } V'\\
h_2 &\mbox{on  }  \widehat{\C} \setminus \overline{\D}\\
\end{array}\right.
$$
Since the map $h_2$ is an external map of $g$, the
map $\tilde{g}$ is holomorphic. Let $\widehat{\varphi}: S \rightarrow \widehat{\C}$ be an isomorphisim that sends the parabolic fixed
point of $\widetilde{g}$ to
infinity, the critical point of  $\widetilde{g}$  to $z=-1$, and the preimage of the parabolic fixed point of
$\widetilde{g}$ to $z=0$. Define $P_2 = \widehat{\varphi} \circ \tilde{g} \circ
\widehat{\varphi}^{-1} : \widehat{\C} \rightarrow \widehat{\C}$. The map $P_2$ is a degree $2$
holomorphic map defined on the Riemann sphere, so it is a quadratic rational function.
By construction it has a parabolic fixed point of multiplier $1$ at $z=\infty$ with preimage
$z=0$, and it has a critical point at $z=-1$.
Hence $P_2$ belongs to the family $Per_1(1)$.
\end{proof}

\begin{teor}\label{thm}
Let ($f, U', U, \gamma_f$) be a parabolic-like mapping of some degree $d>1$, and
 $h: \S^1 \rightarrow \S^1$ be a parabolic external map of the same degree $d$. Then there exists a
parabolic-like mapping ($g, V', V, \gamma_g$) which is hybrid equivalent to $f$ and whose external class is
$[h]$.
\end{teor}
Throughout this proof we assume, in order to simplify the notation, $U$ and $U'$ with $C^1$ boundaries 
(if $U$ and $U'$ do not have $C^1$ boundaries we consider a parabolic-like restriction of ($f, U', U,
\gamma_f$) with $C^1$ boundaries).

Let $h: \S^1 \rightarrow \S^1$ be a parabolic external map of degree $d>1$, $z_*$ be its parabolic fixed point 
and $h:W' \longrightarrow W$ be a degree $d$ covering extension.
Define $B= W \cup \D$ and $B' = W' \cup \D$. 
We are going to construct now a dividing arc $\widetilde \gamma:[-1,1]\rightarrow \overline B
\setminus \D$ for $h$, such that on $\widetilde \gamma$ the dynamics of $h$ is conjugate to the dynamics of
$f$.

Let $h_f$ be an external map of $f$, $z_1$ its parabolic fixed point, $h_f:W_f' \rightarrow W_f$ a degree $d$ covering extension
and $\alpha$ an external equivalence between $f$
and $h_f$. The dividing
arcs
$\gamma_{h_f \pm}$ are tangent to $\S^1$ at the parabolic fixed point $z_1$, and they
divide $W_f$ and $W_f'$ in $\Delta_W, \Omega_W$ and $\Delta'_W, \Omega'_W$ respectively (see Section \ref{propex}).

Let $\Xi_{h_f \pm}$ be repelling petals for the parabolic fixed point $z_1$ which intersect the unit circle and
$\phi_{\pm}: \Xi_{h_f \pm} \rightarrow \H_l  $ be Fatou coordinates with axis tangent to the unit
circle at the parabolic fixed point $z_1$. 
On the other hand, let $\Xi_{h \pm}$ be repelling petals for the parabolic fixed point $z_*$ of $h$ which intersect the
unit circle and
$\widetilde \phi_{\pm}: \Xi_{h \pm} \rightarrow \H_l  $ be Fatou coordinates with axis tangent to the unit
circle at the parabolic fixed point $z_*$.
Define  $$\widetilde{\gamma}_{+}= \widetilde{\phi}_{+}^{-1}(\phi_{h_f +}(\gamma_{h_f +}))$$ and 
$$\widetilde{\gamma}_{-}= \widetilde{\phi}_{-}^{-1}(\phi_{h_f -}(\gamma_{h_f -})).$$ 
The arc $\widetilde{\gamma}= \widetilde{\gamma}_+ \cup
\widetilde{\gamma}_-$ is (possibly after rescaling) a dividing arc for $h$. It
divides the set $B$ into $\Omega_B$ and $\Delta_B$ (with $\D \in \Omega_B$) and the set $B'$ into
$\Omega'_B$ and $\Delta'_B$ (with $\D \in \Omega'_B$).
Define the map $\widetilde{\phi}^{-1}\circ \phi_{h_f}: \gamma_{h_f} \rightarrow \widetilde{\gamma}$ as follows:
$$ \widetilde{\phi}^{-1}\circ \phi_{h_f}(z)=\left\{
\begin{array}{cl}
\widetilde{\phi}_{+}^{-1}\circ \phi_{h_f +} &\mbox{on  } \gamma_{h_f +}\\
\widetilde{\phi}_{-}^{-1}\circ \phi_{h_f -} &\mbox{on  } \gamma_{h_f -}\\
\end{array}\right.
$$  
By Proposition \ref{arcprop}(2) the map $\widetilde{\phi}^{-1}\circ \phi_{h_f}$ is a quasisymmetric conjugacy between $h_{f|\gamma_{h_f}}$
and $h_{|\widetilde{\gamma}}$.
Let $z_0$ be the parabolic fixed point of $f$, and define the map $\psi: \gamma_f \rightarrow \widetilde{\gamma}$ as
follows:
$$ \psi(z)=\left\{
\begin{array}{cl}
\widetilde{\phi}_{+}^{-1}\circ \phi_{h_f +} \circ \alpha &\mbox{on  } \gamma_{f +} \setminus \{z_0\}\\
\widetilde{\phi}_{-}^{-1}\circ \phi_{h_f -} \circ \alpha &\mbox{on  } \gamma_{f -} \setminus \{z_0\}\\
z_* &\mbox{on  } z_0\\
\end{array}\right.
$$  
The map $\psi: \gamma_f \rightarrow \widetilde{\gamma}$ is an orientation preserving homeomorphism, real-analytic on
$\gamma_f \setminus \{z_0\}$, which conjugates the dynamics of $f$ and $h$.
Let $\psi_0:\partial U \rightarrow \partial B$ be an 
orientation preserving $C^1$-diffeomorphism coinciding with $\psi$ on $\gamma_f \cap \partial U$ (it exists because
both $U$ and $B$ have smooth boundaries).
\begin{claim}
 There exists a quasiconformal map $\Phi_{\Delta}:\Delta \rightarrow  \Delta_B$ which extends to 
$\psi$ on $\gamma_{f}$, and to $\psi_0$ on $\partial U \cap \partial \Delta$.
\end{claim}
\begin{proof}
It is sufficient to construct a quasiconformal map $\Phi_{\Delta_W}:\Delta_W \rightarrow  \Delta_B$ which extends to
$\widetilde{\phi}^{-1}\circ \phi_{h_f}$ on $\gamma_{h_f}$ and to $\psi_0\circ \alpha^{-1}$ on
$\alpha(\partial U \cap \partial \Delta)$. Then we will set 
$\Phi_{\Delta}= \Phi_{\Delta_W}\circ \alpha$.

The set $\partial \Delta_W$ is a quasicircle, since it is a piecewise $C^1$ closed
curve with non-zero interior angles. Indeed, $\gamma_{h_f +}$ and $\gamma_{h_f -}$ form a positive angle since they
are separated by at least one attracting petal, and we can assume the angles between $\gamma_{h_f}$
and $\partial W_f$ to be positive (we may take parabolic-like restrictions).
The same argument shows that  $\partial \Delta_B$ is a quasicircle.
Let $\Phi_f:\Delta_W \rightarrow \D $ and $\Phi_h:\Delta_B \rightarrow \D$ be Riemann maps, and let $\Psi_f:\D
\rightarrow \Delta_W$ and $\Psi_h:\D \rightarrow \Delta_B$ be their inverse maps. By the Carathéodory Theorem the maps
$\Psi_f$ and $\Psi_h$ extend continuously to the boundaries, and since $\partial \Delta_W$ and $\partial \Delta_B$ are
quasicircles, the restrictions 
$\Psi_f: \S^1\rightarrow \partial \Delta_W$ and $\Psi_h: \S^1\rightarrow \partial \Delta_B$ are quasisymmetric.
Define the map $\widetilde \Phi_0: \S^1 \rightarrow \S^1$ as follows:
$$ \widetilde{\Phi_0}(z)=\left\{
\begin{array}{cl}
\Psi_h^{-1} \circ  \widetilde{\phi}^{-1}\circ \phi_{h_f} \circ \Psi_f &\mbox{on  } \Psi_f^{-1}(\gamma_{h_f})\\
\Psi_h^{-1} \circ  \psi_0 \circ \alpha^{-1}\circ \Psi_f &\mbox{on  } \Psi_f^{-1} (\partial \Delta_W \cup \partial W_f)\\
\end{array}\right.
$$  
The map $\widetilde \Phi_0: \S^1 \rightarrow \S^1$ is quasisymmetric, because
the extensions of $\Psi_{h}$ and $\Psi_{f}$ to the unit circle are quasisymmetric,
$\alpha$ is conformal, the map $\psi_0$ is a $C^1$-diffeomorphism and by Proposition \ref{arcprop}(2)
the map
$\widetilde{\phi}^{-1}\circ \phi_{h_f}:\gamma_{h_f} \rightarrow \widetilde\gamma$ is
quasisymmetric. Hence it
extends by the Douady-Earle extension (see \cite{DE}) to a quasiconformal map
$\widetilde \phi_0:\overline \D \rightarrow \overline \D$ which is a real-analytic diffeomorphism
on $\D$. Thus $\Phi_{\Delta_W}:= \Psi_h \circ \widetilde \phi_0 \circ \Phi_f$ is a
quasiconformal map between $\overline{\Delta_W}$ and $\overline{\Delta_B}$, which is a real-analytic diffeomorphism on
$\Delta_W$, and which coincides with $ \widetilde{\phi}^{-1}\circ\phi_{h_f}$ on
$\gamma_{h_f}$ and with $\psi_0\circ \alpha^{-1}$ on
$\alpha(\partial U \cap \partial \Delta)$.
\end{proof}
Let us define $\widetilde\Delta_B= h(\Delta_B \cap \Delta'_B),\,\,\, \widetilde B = \Omega_B \cup \widetilde \gamma
\cup \widetilde\Delta_B,\,\,\widetilde B' = h^{-1}(\widetilde B),\,\,\,\widetilde \Omega_B'=
\Omega_B' \cap \widetilde B',\,\,\,\widetilde \Delta_B'=\Delta_B' \cap \widetilde B'$.
On the other hand define $\widetilde{\Delta}= \Phi_{\Delta}^{-1}(\widetilde \Delta_B)$, $\widetilde{\Delta'}=
\Phi_{\Delta}^{-1}(\widetilde \Delta_B')$, $\widetilde{U}= (\Omega  \cup \gamma_f \cup \widetilde{\Delta}) \subset U$.
Consider
$$ \widetilde{f}(z)=\left\{
\begin{array}{cl}
\Phi_{\Delta}^{-1} \circ h \circ \Phi_{\Delta} &\mbox{on  } \widetilde{\Delta'} \\
f  &\mbox{on  }  \Omega'\cup \gamma_f\\
\end{array}\right.
$$
Define $\widetilde{U'}=\widetilde{f}^{-1}(\widetilde{U})$, and $\widetilde{\Omega'}=\widetilde{U'} \cap \Omega'$.
The map $ \widetilde{f}: \widetilde{U'}\rightarrow \widetilde{U} $ is a degree $d$ proper and quasiregular map which
coincides with $f$ on $(\widetilde{\Omega'}\cup \gamma_f) \subset (\Omega'\cup \gamma_f)$.
Define $\widehat {U'}=f^{-1}(\widetilde{U})$, $\widehat {\Delta'}= \Delta' \cap \widehat {U'}$ and $\widehat {\Omega'}=
\Omega' \cap \widehat {U'}$. Then $(f, \widehat {U'}, \widetilde {U}, \gamma_f)$ is a parabolic-like restriction of $(f,
U', U, \gamma_f)$, and $\widehat {\Omega'}= \widetilde{\Omega'}$.
Set $Q_f=\Omega \setminus \overline{\widetilde{\Omega'}}$, and $Q_h= \Omega_B \setminus \overline{\widetilde \Omega'}_B$. Let
$\overline{\psi}_0: \partial \widetilde U \rightarrow \partial \widetilde B$ be an 
orientation preserving $C^1$-diffeomorphism coinciding with $\psi_0$ on $\partial \Omega$, and let $\psi_1:\partial
\widetilde{U'} \rightarrow \partial\widetilde B'$ be
a lift of $\overline \psi_0 \circ \widetilde{f}$ to $h$.
\begin{claim}
 There exists a homeomorphism $\widetilde{\psi}:\overline U \setminus \widetilde\Omega \rightarrow
\overline B \setminus \widetilde\Omega_B$ quasiconformal on $\overline U \setminus (\widetilde\Omega \cup \{z_0\})$
such that the almost complex structure $\sigma$ defined as:
$$ \sigma(z)=\left\{
\begin{array}{cl}
\sigma_0 &\mbox{on  } K_f\\
\sigma_1 = \widetilde{\psi}^*(\sigma_0)  &\mbox{on  } \overline U \setminus \overline{\widetilde\Omega}\\
(\widetilde{f}^n)^*\sigma_1 &\mbox{on  } \widetilde{f}^{-n}(\overline{\widetilde U} \setminus \overline{\widetilde\Omega})\\
\end{array}\right.
$$
is bounded and $\widetilde{f}$-invariant.
\end{claim}
\begin{proof}
 Let us start by constructing a quasiconformal map $\Psi_Q$ between $\overline{Q_f}$ and 
$\overline{Q_h}$ which agrees with $\psi$ on $\gamma_f$, with $\psi_0$ on $\partial U$ and with $\psi_1$ on
$\partial \widetilde{U'}$. 
The sets $\partial Q_f$ and $\partial Q_h$ are quasicircles, since they are piecewise $C^1$ closed
curves with non-zero interior angles. Let $\varphi_f: Q_f \rightarrow \D$ and $\varphi_h: Q_h \rightarrow \D$
be Riemann maps, and let $\psi_f:\D
\rightarrow Q_f$ and $\psi_h:\D \rightarrow Q_h$ be their inverses. By the Carathéodory Theorem
$\psi_f$ and $\psi_h$ extend continuously to the boundaries, and since 
$\partial Q_f$ and $\partial Q_h$ are quasicircles, $\psi_{f_{|\S^1}}$ and $\psi_{h_{|\S^1}}$ are quasisymmetric.
Hence the map $\widehat \Phi_0: \S^1 \rightarrow \S^1$ defined as:
$$ \widehat{\Phi_0}(z)=\left\{
\begin{array}{cl}
\psi_h^{-1} \circ  \psi_+ \circ \psi_f &\mbox{on  } \psi_f^{-1}(\gamma_{f +})\\
\psi_h^{-1} \circ  \psi_0 \circ \psi_f &\mbox{on  } \psi_f^{-1}(\overline{Q_f} \cap \partial U)\\
\psi_h^{-1} \circ  \psi_- \circ \psi_f &\mbox{on  } \psi_f^{-1}(\gamma_{f -})\\
\psi_h^{-1} \circ  \psi_1 \circ \psi_f &\mbox{on  } \psi_f^{-1}(\overline{Q_f} \cap \partial \widetilde U')\\
\end{array}\right.
$$  
is quasisymmetric, and it
extends (see \cite{DE}) to a quasiconformal map
$\widehat \varphi_0:\overline \D \rightarrow \overline \D$ which is a real-analytic diffeomorphism
on $\D$. Finally, the map $\Psi_{Q}:= :\psi_h\circ \widehat{\varphi_0} \circ \psi_f^{-1}:\overline{Q_f} \rightarrow \overline{Q_h}$ is a 
quasiconformal map which coincides with  $\psi$ on $\gamma_f$, with $\psi_0$ on $\partial U$ and with $\psi_1$ on
$\partial \widetilde{U'}$. Moreover the map $\Psi_{Q}$ is a real-analytic diffeomorphism on
$Q_f$.

Define the homeomorphism $\widetilde{\psi}:\overline U \setminus \widetilde\Omega \rightarrow
\overline B \setminus \widetilde\Omega_B$ quasiconformal on $\overline U \setminus (\widetilde\Omega \cup \{z_0\})$
as follows:
$$ \widetilde{\psi}(z)=\left\{
\begin{array}{cl}
\psi  &\mbox{on  } \gamma_f\\ 
\psi_0 &\mbox{on  } \partial U \\
\psi_1 &\mbox{on  } \partial \widetilde \Omega' \cap \partial \widetilde U'\\
\Psi_Q &\mbox{on  } Q_f \\
\Phi_{\Delta} &\mbox{on  } \Delta\\
\end{array}\right.
$$ 
Therefore the almost complex structure
$$ \sigma(z)=\left\{
\begin{array}{cl}
\sigma_0 &\mbox{on  } K_f\\
\sigma_1 = \widetilde{\psi}^*(\sigma_0)  &\mbox{on  } \overline U\setminus\overline {\widetilde \Omega'}\\
(\widetilde{f}^n)^*\sigma_1 &\mbox{on  } \widetilde{f}^{-n}(\overline{\widetilde U} \setminus \overline{\widetilde\Omega})\\
\end{array}\right.
$$
is bounded and $\widetilde{f}$-invariant. 
\end{proof}
By the Measurable Mapping Theorem,
there exists a quasiconformal map $\varphi:U \rightarrow  \C$ such that $\varphi^* \sigma_0=\sigma$.
Let $$g:=\varphi \circ \widetilde{f} \circ \varphi^{-1}\,: \varphi(\widetilde{U'}) \rightarrow
\varphi(\widetilde{U}).$$
Let us call $V'=\varphi(\widetilde{U'})$, $V=\varphi(\widetilde{U})$,
$\gamma_{g+}=\varphi(\gamma_{f +})$ and
$\gamma_{g-}=\varphi(\gamma_{f -})$. Then ($g, V', V, \gamma_g$) is a
parabolic-like map hybrid equivalent to $f$. Indeed, since $\widetilde{f}_{|\widetilde{\Omega'}\cup \gamma_f}=f$,
$\widehat \Omega= \widetilde \Omega$ and $(f, \widehat {U'}, \widetilde {U}, \gamma_f)$ is a parabolic-like restriction of $(f, U', U, \gamma_f)$,
the map $\varphi$ is a quasiconformal conjugacy between $f$ and $g$, and
$\varphi^* \sigma_0=\sigma_0$ on
$K_f$ by construction.

If $K_f$ is connected, define the quasiconformal map $\widehat{\psi}:U \setminus K_f \rightarrow
B \setminus \overline{\D}$ as follows:
$$ \widehat{\psi}(z)=\left\{
\begin{array}{cl}
\widetilde{\psi} &\mbox{on  } U\setminus(\widetilde{\Omega'} \cup \{z_0\})\\
h^{-n} \circ \widetilde{\psi} \circ \widetilde{f}^{n} &\mbox{on  }  \widetilde{f}^{-n}(\widetilde{U}\setminus \widetilde \Omega')\\
\end{array}\right.
$$
Then the quasiconformal map
$\overline{\psi}=\widehat{\psi} \circ \varphi^{-1}:(V \cup
V')\setminus K_g \rightarrow B \setminus \overline{\D}$ is an external equivalence between $g$ and
$h$, since by construction on $V'\setminus K_g$ $\overline{\psi} \circ g = h \circ
\overline{\psi}$, and $\overline{\psi}$ is holomorphic
(indeed $(\widehat{\psi} \circ \varphi^{-1})^*\sigma_0=\sigma_0$).

If $K_f$ is not connected, let $V_f\approx \D$ be a full relatively compact connected subset of
$\widetilde{U}$, containing $\overline {\widetilde{\Omega'}}$, the critical values of $\widetilde{f}$ and such that
($f, f^{-1}(V_f), V_f,\gamma_f$) is a parabolic-like restriction of
$(f,\,\widehat{U'},\,\widetilde{U},\,\gamma_f)$.
Call $L=\widetilde{f}^{-1}(\overline V_f)\cap \overline {\widetilde{\Omega'}}$.
Define the map $\widehat{\psi}: U\setminus L \rightarrow B \setminus \overline{\D}$
as follows:
$$ \widehat{\psi}(z)=\left\{
\begin{array}{cl}
\widetilde{\psi} &\mbox{on  } U\setminus(\widetilde\Omega' \cup \{z_0\})\\
h^{-1} \circ \widetilde{\psi} \circ \widetilde{f} &\mbox{on  }  \widetilde{f}^{-1}(\widetilde U\setminus\widetilde \Omega')\setminus L\\
\end{array}\right.
$$
Let $V_g\approx \D$ be a full relatively
compact connected subset of $V$ containing
$\overline \Omega_g'$, the critical values of $g$ and such that $(g,g^{-1}(V_g),V_g,\gamma_g)$ is a parabolic-like
restriction of $(g,\,V,\,V',\,\gamma_g)$. Call $M=g^{-1}(\overline V_g)\cap \overline \Omega_g'$.
Then the map 
$\overline{\psi}=\widehat{\psi} \circ \varphi^{-1}:(V \cup V')\setminus M \rightarrow B \setminus \overline{\D}$ is an
external equivalence between $g$ and $h$ (cfr. Lemma \ref{lem}).
\subsection{Unicity}

\begin{prop}\label{hol}
 Let $f: U'\rightarrow U$ and $g :V' \longrightarrow V$ be two parabolic-like mappings of degree $d$ with connected
Julia sets. If 
they are hybrid and externally equivalent, then they are holomorphically equivalent.
\end{prop}

\begin{proof}
Let $\varphi:A \rightarrow B$ be a hybrid equivalence between $f$ and $g$, and $\psi: (A_1 \cup A_1')\setminus K_f
\rightarrow (B_1 \cup B_1')\setminus K_g$ an external equivalence between $f$ and $g$. Let $h: \widetilde W' \rightarrow \widetilde W$ be an
external map of $f$ constructed from the Riemann map $\alpha: \widehat \C \setminus K_{f} \rightarrow \widehat\C \setminus
\overline{\D}$. Let $A_f$ be a topological disc compactly contained in $(A_1 \cup A_1') \cap A$ and such that $B_{\varphi}=
\varphi(A_f)$ is compactly contained in $(B_1 \cup B_1')$, and $B_{\psi^o}=\psi(A_f\setminus K_f)$
is compactly contained in $B$. Set $W_{\beta}= \alpha \circ \psi^{-1}( B_{\varphi} \setminus K_g) $.
The map $\beta= \alpha \circ \psi^{-1}: B_{\varphi} \setminus K_g \rightarrow W_{\beta}$ restricts to an external
equivalence between $g$ and $h$.
 
Define $B_{\psi}= B_{\psi^o} \cup K_g$ and the
map $\Phi: A_f  \rightarrow B_{\psi}$ as:
$$ \Phi(z)=\left\{
\begin{array}{cl}
 \varphi &\mbox{on  } K_{f}\\
 \psi &\mbox{on  }  A_f \setminus K_f\\
\end{array}\right.
$$

By construction the map $\Phi: A_f \rightarrow B_{\psi}$ conjugates the
maps 
$f$ and $g$ conformally on $A_f$ and quasiconformally with
$\overline{\partial}\Phi=0$ on $K_{f}$.
We want to prove that the map $\Phi$ is holomorphic. By Rickmann Lemma (see below) $\Phi$ is holomorphic if $\Phi$ is
continuous. Thus we just need to prove that it is continuous.

Define
$W_f=h(h^{-1}(\alpha(A_f\setminus K_f)) \cap \alpha(A_f\setminus K_f))\subset \alpha(A_f\setminus K_f)$ 
and $W'_f=h^{-1}(W_f)$. The
restriction $h:W'_f \rightarrow W_f $ is proper holomorphic and of degree $d$.
The map $\chi:=\beta \circ \varphi \circ \alpha^{-1} : W_f'
\rightarrow W_{\beta}$ is a quasi-conformal homeomorphism (into its image) which autoconjugates
$h$ on 
$\Omega'_{W} \cup \gamma_h\setminus \{\gamma_{h}(0)\}$.

Setting $\tau(z)=1/\bar{z}$, $\widetilde W_f'= W'_f \cup \S^1 \cup \tau(W_f')$, $\widetilde W_{\beta}= W_{\beta} \cup \S^1 \cup \tau(W_{\beta})$,
and applying the
strong
reflection principle with respect to
the unit circle, we obtain a quasiconformal homeomorphism (into its image)
$\widetilde \chi: \widetilde W_f'\rightarrow \widetilde  W_{\beta}$, 
which autoconjugates $h$ on $\widetilde\Omega'_{W}$.
Thus the restriction
$\widetilde \chi: \S^1
\rightarrow \S^1$ is a quasisymmetric autoconjugacy of $h$ on the
unit circle. Since the preimages of the parabolic fixed point $z=1$ are dense in $\S^1$, an
autoconjugacy of $h$ on the
unit circle is the identity. Therefore $\widetilde \chi \mid_{\S^1} = Id $. 

Since the map $\widetilde \chi:\widetilde W_f'\rightarrow \widetilde W_{\beta}$
is a quasiconformal homeomorphism which
coincides with the identity on $\S^1$, the hyperbolic distance between a point near $\S^1$ and its image is uniformly
bounded, i.e. $\exists M>0$ and $r>1$ such that:
$$\forall z\, ,\, 1<|z|<r,\,\,\,d_{W'_{f}} (z, \beta \circ \varphi \circ \alpha^{-1} (z)) \leq M.
$$ 
Since $\alpha$ and $\beta$ are isometries, we obtain 
$$d_{A_f \setminus K_{f}} (\beta^{-1} \circ \alpha(z), \varphi(z)) \leq
M \mbox{         for   
 } z \notin K_{f},\,z \mbox{ in a neighborhood of }K_f.$$ 
Then $\beta^{-1} \circ \alpha (z)$ and $\varphi(z)$ converge to the same value as $z$
converges to $J_{f}$, i.e. $\beta^{-1} \circ \alpha$ extends continuously to $J_f$ by $\beta^{-1} \circ \alpha
(z)=\varphi(z),\,\, z \in J_f$.
 Thus $\Phi$ is continuous. The result follows by the Rickmann lemma (for a proof of the Rickmann lemma we refer to
\cite{DH}, Lemma $2$ pg. 303):

\begin{lemma}\textit{\textbf{Rickmann}}
Let $U\subset \C$ be open, $K\subset U$ be compact, $\varphi: U \rightarrow \C$ and $\Phi: U \rightarrow \C$ be two maps 
which are homeomorphisms onto their images. Suppose that $\varphi$ is quasi-conformal, that $\Phi$ is quasi-conformal on $
U \setminus K$ and that $\Phi=\varphi$ on $K$. Then $\Phi$ is quasiconformal and $D\Phi = D \varphi$ almost everywhere on $K$.

\end{lemma}
\end{proof}
\begin{prop}\label{unic}
 If $P_A=z+1/z+A$ and $P_{A'}=z+1/z+A'$ are hybrid conjugate and $K_A$ is connected, then they are
holomorphically conjugate, i.e. $A^2=(A')^2$.
\end{prop}
\begin{proof}
 Since $K_A$ and $K_{A'}$ are connected, the external conjugacies between $P_A$ and $P_{A'}$ respectively and $h_2$ can
be extended to the discs $\widehat{\C} \setminus K_A$ and $\widehat{\C} \setminus K_{A'}$ (see Proposition \ref{extmap}), i.e.
there exist holomorphic conjugacies $\alpha:\widehat{\C} \setminus K_A \rightarrow \widehat{\C} \setminus
\overline{\D}$ and $\beta:\widehat{\C} \setminus K_{A'} \rightarrow \widehat{\C} \setminus \overline{\D}$ between $P_A$
and $P_{A'}$ respectively and $h_2$. Therefore $\beta^{-1} \circ \alpha: \widehat{\C} \setminus K_{A'} \rightarrow
\widehat{\C} \setminus  K_{A'}$ is a holomorphic conjugacy between $P_A$ and $P_{A'}$.

Let $(f, U', U, \gamma_f)$ and $(g, V', V, \gamma_g)$ be parabolic-like restrictions of $P_A$ and $P_{A'}$
respectively, and let $\varphi:A \rightarrow B$ be a hybrid equivalence between them. Define the map $\Phi: \widehat{\C}
\rightarrow \widehat{\C}$ as follows:
$$ \Phi(z)=\left\{
\begin{array}{cl}
 \varphi &\mbox{on  } K_{A}\\
 \beta^{-1} \circ \alpha &\mbox{on  }  \widehat{\C} \setminus K_{A}\\
\end{array}\right.
$$

The proof of Proposition \ref{hol} shows that the map $\Phi: \widehat{\C} \rightarrow \widehat{\C}$ is holomorphic, hence it is
a M\"{o}bius transformation. Since $\Phi$ conjugates $P_A$ and $P_{A'}$, it
fixes the parabolic fixed point $z=\infty$ and its preimage $z=0$, and it
can fix or interchange the critical points $z=1$ and $z=-1$. Hence either $\Phi$ is the Identity or $\Phi(z)$ is the map $z \rightarrow -z$.
Therefore $[P_A]= \{ P_A,\,\,P_{-A}\}$, and finally $A^2=(A')^2$.

\end{proof}

\end{document}